\documentclass[12pt]{amsart}
\numberwithin{equation}{section}
\usepackage{amsmath,amsthm,amsfonts,amscd,eucal,mathabx}

\usepackage{xypic}
\usepackage{graphicx}
\usepackage{amssymb}
\def\cb{{\mathcal B}}
\def\cc{{\mathcal C}}

\def\ce{{\mathcal E}}

\def\ch{{\mathcal H}}

\def\cam{{\mathcal M}}
\def\cn{{\mathcal N}}

\def\cp{{\mathcal P}}

\def\cs{{\mathcal S}}

\def\cu{{\mathcal U}}

\def\ga{{\mathfrak A}}
\def\gb{{\mathfrak B}}

 \def\gpn{{\mathfrak n}}

\def\ba{{\mathbb A}}

\def\bc{{\mathbb C}}

\def\bn{{\mathbb N}}

\def\br{{\mathbb R}}
\def\bt{{\mathbb T}}

\def\bz{{\mathbb Z}}

\def\a{\alpha}
\def\b{\beta}
\def\g{\gamma}  
\def\d{\delta}  \def\D{\Delta}

\def\eps{\varepsilon}

\def\l{\lambda} 

\def\m{\mu}

\def\n{\nu}
\def\r{\rho}
\def\s{\sigma} 
\def\t{\tau}
\def\f{\varphi}  \def\F{\Phi}
\def\th{\theta} 
\def\om{\omega} \def\Om{\Omega}

\def\id{\hbox{id}}
\def\ker{{\rm Ker}}

\newtheorem{thm}{Theorem}[section]
\newtheorem{lem}[thm]{Lemma}

\newtheorem{cor}[thm]{Corollary}
\newtheorem{prop}[thm]{Proposition}
\newtheorem{rem}[thm]{Remark}
\newtheorem{defin}[thm]{Definition}

\theoremstyle{definition}
\newtheorem{examp}{Example}[section]

\def\aut{\mathop{\rm Aut}}
\def\ed{\mathop{\rm End}}

\def\ob{\mathop{\rm Obj}}

\newcommand{\ty}[1]{\mathop{\rm {#1}}}
\def\di{{\rm d}}

\def\ad{\mathop{\rm ad}}

\def\idd{{1}\!\!{\rm I}}

\DeclareMathAlphabet{\mathpzc}{OT1}{pzc}{m}{it}

\begin{document}
\title[skew-product dynamical systems]
{skew-product dynamical systems for crossed product $C^*$-algebras and their ergodic properties}
\author[S. Del Vecchio]{Simone Del Vecchio}
\address{Simone Del Vecchio\\
Institut f\"{u}r Theoretische Physik, ITP- Universit\"{a}t Leipzig, Germany}\email{{\tt
simone.del\textunderscore vecchio@physik.unilepzig.de}}
\author[F. Fidaleo]{Francesco Fidaleo}
\address{Francesco Fidaleo\\
Dipartimento di Matematica \\
Universit\`{a} di Roma Tor Vergata\\
Via della Ricerca Scientifica 1, Roma 00133, Italy} \email{{\tt
fidaleo@mat.uniroma2.it}}
\author[S. Rossi]{Stefano Rossi}
\address{Stefano Rossi\\
Dipartimento di Matematica \\
Universit\`{a} degli Studi di Bari Aldo Moro \\
Via Edoardo Orabona 4, Bari 70125, Italy} \email{{\tt
stefano.rossi@uniba.it}}
%\date{\today}

\keywords{Operator Algebras, Noncommutative Harmonic Analysis, Noncommutative Torus, Skew-Product, Ergodic Dynamical Systems, Unique Ergodicity}
\subjclass[2000]{37A55, 43A99, 46L30, 46L55, 46L65, 46L87}

\begin{abstract}

Starting from a discrete $C^*$-dynamical system $(\mathfrak{A}, \theta, \omega_o)$, we define and study most of the main ergodic properties of the crossed product $C^*$-dynamical system 
$(\mathfrak{A}\rtimes_\alpha\mathbb{Z}, \Phi_{\theta, u},\om_o\circ E)$, $E:\mathfrak{A}\rtimes_\alpha\mathbb{Z}\rightarrow\ga$ being the canonical conditional expectation of 
$\mathfrak{A}\rtimes_\alpha\mathbb{Z}$ onto $\mathfrak{A}$, provided $\a\in\aut(\ga)$ commute with the $*$-automorphism $\th$ up tu a unitary $u\in\ga$. Here, 
$\Phi_{\theta, u}\in\aut(\mathfrak{A}\rtimes_\alpha\mathbb{Z})$ can be considered as the fully noncommutative generalisation of the celebrated skew-product defined by H. Anzai for the product of two tori in the classical case.

\end{abstract}

\maketitle

%\tableofcontents

\section{Introduction}

Classical ergodic theory is a well-established topic whose early motivations came undoubtedly from an attempt to
give a sound foundation to statistical mechanics. 
It is exactly in this spirit that the first ergodic theorems were proved.

Noncommutative  ergodic theory is a decidedly younger area of research, whose motivation is more mathematical in character. However, when it comes to producing truly new examples, the noncommutative theory turns out to be much harder to deal with. On the contrary, in the classical theory it is not too difficult to come up with suitably chosen classes of dynamical systems featuring quite a wide spectrum of phenomena.

The so-called Anzai skew-products are certainly a case in point.
First introduced in \cite{A}, and later dealt with in far more generality in \cite{Fu} and called there {\it processes on the torus}, these are also known simply as skew-products, and are homeomorphisms defined on a product space of the type $X_o\times\mathbb{T}$, where $X_o$ is a compact metrizable space and $\mathbb{T}$ the one-dimensional torus.
Corresponding to  any homeomorphism $\theta\in\textrm{Homeo}(X_o)$ and a continuous function
$f\in C(X_o;\mathbb{T})$, one can define a homeomorphism $\Phi_{\theta, f}$ on the product $X_o\times\mathbb{T}$ acting as $\Phi_{\theta, f}(x, z):=(\theta(x), f(x)z)$, $(x, z)\in X_o\times\mathbb{T}$.
If $\theta$ is assumed to be  uniquely ergodic, that is there exists only one probability
regular Borel measure ({\it i.e.} a Radon probability measure) $\mu_o$ on $X_o$ which is invariant under $\theta$, then the properties of $\Phi_{\theta, f}$ can be analised
thoroughly.  The behaviour of the skew-product $(X_o\times\mathbb{T}, \Phi_{\theta, f})$ is in fact ruled
by the so-called cohomological equations $g(\theta(x))f^n(x)=g(x)$, $n\in\mathbb{Z}\smallsetminus\{0\}$, in the unknown function
$g$. More precisely,
the system will be ergodic with respect to the product measure $\mu:=\mu_o\times\l$, where $\l$ is
the Lebesgue-Haar measure on the torus, if and only if, for any integer $n$ different
from zero, the above equations only have null solutions. 
Moreover, in this case the product measure is the only
invariant measure of the skew-product. Not having continuous solutions corresponds to what is usually referred to as
topological ergodicity, namely that the constant functions are the only continuous functions.
Anzai skew-product dynamical systems were then intensively studied for several applications to many branches of mathematics, and were used to produce examples of dynamical systems enjoying very special ergodic and spectral properties. 

It is then natural to address the systematic investigation of the skew-product construction in a totally noncommutative setting. The first noncommutative version of the skew-product was defined in \cite{OP} in order to invesitigate  the so-called Rokhlin property for the seminal example of the noncommutative 2-torus, whereas the main ergodic properties for such an example were recently investigated in \cite{DFGR}.

The primary aim of this paper is to extend the study initiated in \cite{Fu} for the classical case and in \cite{DFGR} for the noncommutative torus, to the general situation relative to crossed products. In fact, the key-point is that the noncommutative torus can also be obtained as a crossed product
of $C(\mathbb{T})$ by the discrete group $\mathbb{Z}$ acting on the one-dimensional torus through a rotation
by $2\pi\alpha$. No wonder, crossed products by $\mathbb{Z}$ provide a natural framework where the sought
generalization can be set. More precisely, the data we start from is a triple $(\mathfrak{A}, \theta, \alpha)$, with $\ga$ is a unital $C^*$-algebra and $\theta, \alpha\in\aut(\ga)$ are $*$-automorphisms of $\ga$.
In order to investigate the ergodic properties of the arising skew-product systems, we suppose that 
$(\mathfrak{A}, \theta)$ is a uniquely ergodic $C^*$-dynamical system whose unique invariant state is $\omega_o$. In order to construct the skew-product dynamical system on the 
crossed product $C^*$-algebra, we need that the $*$-automorphism
 $\alpha$ commutes with $\theta$ up to a unitary $u\in\mathfrak{A}$.
 
The output is the crossed product $\mathfrak{A}\rtimes_\alpha\mathbb{Z}$ acted upon by
$\Phi_{\theta, u}\in{\rm Aut}(\mathfrak{A}\rtimes_\alpha\mathbb{Z})$, which is given by
$\Phi_{\theta, u}(a):=\theta(a)$, $a\in\mathfrak{A}$, and $\Phi_{\theta, u}(V)=uV$, where $V\in\mathfrak{A}\rtimes_\alpha\mathbb{Z}$ is the unitary implementation
of $\alpha$. Not only does $\Phi_{\theta, u}$ extend $\theta$ but it is also the more general extension that preserves the subspaces $\mathfrak{A}V^n\subset\mathfrak{A}\rtimes_\alpha\mathbb{Z}$ for
every integer $n$. Ultimately, this is the felicitous circumstance that opens the door to Fourier analysis when studying the ergodic properties of $\Phi_{\theta, u}$.

For a $C^*$-dynamical system $(\ga,\a,\f)$, we denote by $(\pi_\f,\ch_\f,V_{\f,\a},\xi_\f)$ be the covariant GNS representation corresponding to the invariant state $\f$. The ergodicity of the extended automorphisms on the crossed product $C^*$-algebras $\mathfrak{A}\rtimes_\alpha\mathbb{Z}$, with respect to its distinguished invariant states 
$\omega:=\omega_o\circ E$, $E$ being the canonical
expectation of $\mathfrak{A}\rtimes_\alpha\mathbb{Z}$ onto $\mathfrak{A}$, can still be analysed
by means of cohomological equations in the GNS Hilbert space $\ch_{\om_o}$, each of which corresponds to one of the above subspaces. This is established in Proposition \ref{toerph} and Theorem \ref{mier}.

Let the unitaries $u_n\in\ga$ be such that $\Phi_{\theta, u}(V^n)=u_nV^n$. First of all we show ({\it cf.} Proposition \ref{toerp}) that 
the skew-product $C^*$-dynamical system $(\mathfrak{A}\rtimes_\alpha\mathbb{Z},\Phi_{\theta, u})$ is topologically ergodic ({\it i.e.} 
$(\mathfrak{A}\rtimes_\alpha\mathbb{Z})^{\Phi_{\theta, u}}=\bc I$) if and only if the cohomological equations
\begin{equation}
\label{i}
\a^{-n}(u_n)\th(a)=a
\end{equation} 
have no nontrivial solutions $a\in\ga$ for each $n\neq0$.

When the support in the bidual $\ga^{**}$ of $\om_o$ is central, the ergodicity (and then the unique ergodicity) of the the state $\om$ can be established in terms of the absence of nontrivial solutions of cohomological equations directly for elements in $\pi_{\om_o}(\ga)''$, see Theorem \ref{mier} (iv). In fact, the skew-product dynamical system is uniquely ergodic if and only if the equations
\begin{equation}
\label{i0}
\pi_{\om_o}(\a^{-n}(u_n))\ad{}\!_{V_{\om_o,\th}}(a)=a\,,
\end{equation} 
have no nontrivial solutions when $n\neq0$, directly for $a\in\pi_{\om_o}(\ga)''$. Indeed, under the assumption that $\omega_o$ has central support, the equations we obtain are
of the form \eqref{i0}, where the unknown is now a bounded operator
$a$ (indeed a multiple of a single unitary) that sits in the von Neuman algebra $\pi_{\omega_0}(\mathfrak{A})''$. However, as an effect of $\mathfrak{A}$ no longer being  necessarily commutative, the celebrated Murray-von Neumann BT Theorem ({\it cf.} \cite{MvN}, Lemma 9.2.1, see also \cite{FZ}, Theorem 4.1) cannot be used to pass from
$L^2(\pi_{\om_o}(\ga)'')=\ch_{\om_o}$ to $L^\infty(\pi_{\om_o}(\ga)'')=\pi_{\om_o}(\ga)''$, and therefore
we need to employ more advanced techniques from Tomita-Takesaki theory combined with ergodicity, to write the equations directly at the level of the GNS
von Neumann algebra $\pi_{\omega_0}(\ga)''$.
This is done through a careful application of theory of the positive self-polar cone associated with a standard von Neumann algebra, which is briefly recalled in Section 
\ref{thecone} relative to our purposes.

Suppose now that in addition, $\om_o\in\cs(\ga)$ is also faithful.\footnote{In the noncommutative case, there exist non faithful states $\om_o\in\cs(\ga)$ whose support in the bidual is central, see below.} We then have the $*$-isomorphic copy $\pi_{\om_o}(\ga)$ of $\ga$ in $\pi_{\om_o}(\ga)''$. Therefore, it is meaningful to distinguish the solutions 
$a\in\pi_{\om_o}(\ga)$ of \eqref{i0}, which indeed correspond to that of \eqref{i} in this particular situation,
denoted as {\it \lq\lq continuous"}, and those $\pi_{\om_o}(\ga)''\smallsetminus\pi_{\om_o}(\ga)$ denoted as {\it \lq\lq measurable non continuous}.\footnote{The general definition involving the nature of the solutions of the cohomological equations is provided in Definition \ref{dcmn}.} 

The existence/non existence of continuous, or measurable non continuous solutions of the nontrivial solutions of \eqref{i0} is connected with some interesting ergodic properties of the skew-product dynamical systems under consideration, such as the minimality and the unique ergodicity with respect to the fixed-point subalgebra. 

We first recall that, concerning the classical situation when $\ga\sim C(X_o)$ and $\a=\id_\ga$, such that the crossed product collapses to the usual cartesian product between $X_o$ and $\bt$, in \cite{Fu} ({\it cf.} Remark on pag. 582) it was noticed  that the absence of nontrivial continuous solutions of \eqref{i0}, $n\neq 0$,  assures the minimality of the arising skew-product dynamical system. Therefore, for such kind of classical dynamical systems, strict ergodicity would imply minimality.

In addition, in the above mentioned paper by Furstenberg ({\it cf.} pag. 585), an example of a classical skew-product dynamical system for which there are only nontrivial measurable non continuous solutions of \eqref{i0}, $n\neq 0$ was exhibited. Therefore, such an example would provide a $C^*$-dynamical system which is minimal but not uniquely ergodic.

In Section \ref{uniqueergo}, the nature of the solutions of the cohomological equations is also related to the generalisation of unique ergodicity, that is unique ergodicity with respect to the fixed-point subalgebra by Abadie and Dykema in \cite{AD}, which involves the question of whether the norm-limit of the Ces\`{a}ro means, $\lim_{n\rightarrow+\infty}\frac{1}{n}\sum_{k=0}^{n-1} \Phi_{\theta, u}^k(x)$,
exists for any $x\in\mathfrak{A}\rtimes_\alpha\mathbb{Z}$. Indeed, for the uniquely ergodic $C^*$-dynamical system $(\ga,\th,\om_o)$, it is shown that the unique ergodicity with respect to the fixed-point subalgebra implies that the cohomological equations \eqref{i0} have only \lq\lq continuous'' solutions. 

It would be desirable to prove the converse of the last mentioned result. However, we have shown that the converse holds true for processes on the torus: for the classical skew-product examples considered in \cite{Fu} without the assumption of metrisability, the unique ergodicity with respect to the fixed-point subalgebra is equivalent to the existence of only \lq\lq continuous'' solutions of the cohomological equations \eqref{i0}. We also address the interesting, still open, problem concerning the equivalence between minimality and topological ergodicity for skew-product dynamical systems outlined in \cite{Fu}, p. 582.

The present paper is complemented by Section \ref{aind} in which we show how the skew-product dynamical system is similar to the process of {\it alpha-induction} arising from the study of the factor-subfactor inclusions of $W^*$-algebras and the theory of the Jones index. A final section containing some illustrative examples is also added.

\section{Preliminaries}
\label{sec2}

Let $E$ be a normed space with norm $\|\,{\bf\cdot}\,\|$. In particular, $\|x\|\equiv\|x\|_\ch:=\sqrt{\langle x,x\rangle}$ is the Hilbertian norm of $x$ in the Hilbert space $\ch$, equipped with the inner product $\langle\,{\bf\cdot}\,,\,{\bf\cdot}\,\rangle$, linear in the first variable.
For a continuous or a Borel measurable function $f$ defined on the locally compact Hausdorff space $X$ equipped with the Radon measure $\m$, $\|f\|\equiv\|f\|_\infty$ will denote the \lq\lq sup"-norm or the 
\lq\lq esssup''-norm  for $f$ continuous or measurable, respectively. 

\vskip.3cm

Let $\bt:=\{z\in\bc\mid |z|=1\}$ be the abelian group consisting of the unit circle. It is homeomorphic to the interval $[0,2\pi)$ by means of the map $[0,2\pi)\ni\a\mapsto e^{\imath\a}\in\bt$, after identifying the end-points $0$ and $2\pi$. On $\bt$, $\l$ denotes its the Lebesgue-Haar measure.

\vskip.3cm

If $\mathpzc{f}:X\to X$ is an invertible map on a point-set $X$:
\begin{itemize} 
\item[{\bf --}] $\mathpzc{f}^0:=\id_X$, and its inverse is denoted by $\mathpzc{f}^{-1}$;
\item[{\bf --}] for the $n$-times composition, $\mathpzc{f}^n:=\underbrace{\mathpzc{f}\circ\cdots\circ \mathpzc{f}}_{n-\text{times}}$;
\item[{\bf --}] for the $n$-times composition of the inverse, $\mathpzc{f}^{-n}:=\underbrace{\mathpzc{f}^{-1}\circ\cdots\circ \mathpzc{f}^{-1}}_{n-\text{times}}$.
\end{itemize}
Therefore, $\mathpzc{f}^{n}:X\to X$ is meaningful for any $n\in\bz$ with the above convention, and provides an action of the group $\bz$ on the set $X$. 

\vskip.3cm

If $(x_k)_{k\in\bz}\subset X$ is a two-sided sequence in the Banach space $X$, we say that the series 
$\sum_{k\in\bz}x_k$ {\it converges (or is summable) in the sense of Ces\`aro} if the sequence of the Ces\`aro averages 
$\Big(\frac1{n}\sum_{l=0}^{n-1}\big(\sum_{|k|\leq l}x_k\big)\Big)_{n\in\bn}\subset X$ of the partial sums is a Cauchy sequence.\footnote{The analogous definition for one-sided sequence $(x_k)_{k\in\bn}\subset X$ can be easily provided.}

\vskip.3cm

If $\ga$ is a unital $C^*$-algebra, we denote by $\cu(\ga)$ the group of the unitary elements of $\ga$. Analogously for $\ch$ a Hilbert space, with a slight abuse of notations we denote by $\cu(\ch)\subset\cb(\ch)$ the group of the unitary operators acting on $\ch$.

For $a\in\cb(\ch)$ and $\l\in\bc$, denote by $E_\l^a$ the self-adjoint projection onto the closed subspace of $\ch$ of the eigenvectors associated to $\l$. The operator $a\in\cb(\ch)_1$, the unit ball of $\cb(\ch)$, is said to be {\it mean-ergodic} if $\dim(E_1^a)=1$.

\vskip.3cm

For a $C^*$-algebra $\mathfrak{A}$, with or without an identity, and a positive linear functional $f\in\ga^*$, with $(\ch_f,\pi_f,\xi_f)$ we denote the Gelfand-Naimark-Segal (GNS for short) representation associated to $f$, see {\it e.g.} \cite{T}.

For the positive linear functional $f$ on the $C^*$-algebra $\ga$, we set
$$
\gpn_f:=\{x\in\ga\mid f(x^*x)=0\}\,,
$$
which is a closed left ideal of $\ga$. In addition, 
\begin{equation}
\label{idek}
\ker(\pi_f)\subset\gpn_f\,.
\end{equation}
If $\ch$ is a Hilbert space with $\dim(\ch)\geq2$ and $\xi\in\ch$ is a non null-vector with $\om_\xi:=\langle\,{\bf\cdot}\,\xi,\xi\rangle$ the vector functional, we have 
$\ker(\pi_{\om_\xi})\subsetneq\gpn_{\om_\xi}$.

Notice that the cyclic vector $\xi_f$ of the GNS representation of $f$ is also separating for 
$\pi_f(\ga)''$ if and only if its support $s(f)$ (of its canonical embedding) in the bidual $W^*$-algebra $\ga^{**}$ is central, see {\it e.g.} \cite{NSZ}, pag. 15.
At any rate, the property is not automatic. In fact, examples can be given of separating vectors for a given concrete  $C^*$-algebra $\mathfrak{A}\subset\mathcal{B}(\mathcal{H})$ which nonetheless fail to be separating for $\mathfrak{A}''$ , see {\it e.g.} \cite{T0}, Example 1 and Example 2.

\vskip.3cm

We recall some basic definitions concerning ${\rm End}(\mathfrak{A})$, which are used somewhere in the sequel.
\begin{itemize}
\item [(i)] The objects of ${\rm End}(\mathfrak{A})$ are all unital $*$-endomorphisms of $\mathfrak{A}$. Obviously, $\aut(\ga)\subset{\rm End}(\ga)$.
\item [(ii)]  The morphisms of ${\rm End}(\mathfrak{A})$ are given by intertwining elements, that is for any two
objects $\rho, \sigma$ of ${\rm End}(\mathfrak{A})$,  the set ${\rm Hom}(\rho, \sigma)$ of morphisms between $\rho$ and $\sigma$ is the Banach space
$$
{\rm Hom}(\rho, \sigma):=\{x\in\ga\mid\rho(a)x=x\sigma(a)\,,\,\,\,a\in\ga\}\,.
$$
\item [(iii)] Two objects $\rho, \sigma$ are said to be unitarily equivalent if there exists a unitary $u\in{\rm Hom}(\rho, \sigma)$.
\item [(iv)] The tensor product of objects in ${\rm End}(\mathfrak{A})$ is the composition of endomorphisms, and the tensor product between arrows $x_1\in{\rm Hom}(\rho_1,\sigma_1)$,  $x_2\in{\rm Hom}(\rho_2,\sigma_2)$ is defined as
$$
{\rm Hom}(\r_1\otimes\rho_2,\s_1\otimes\sigma_2)\ni x_1\otimes x_2:=x_1 \s_1(x_2)=\r_1(x_2)x_1\,.
$$
\end {itemize}

\section{$C^*$-dynamical systems}

A (discrete) $C^*$-dynamical system is a pair $(\ga,\F)$, where $\ga$ is a $C^*$-algebra and $\F:\ga\to\ga$ a completely positive linear map. We also suppose that
$\|\F\|=1=\sup_\iota\|\F(u_\iota)\|$,
where $(u_\iota)_\iota\subset\ga$ is any approximate identity. 
If $\ga$ has the identity $\idd_\ga$, we additionally suppose that $\F$ is unital, and therefore
$\|\F\|=\|\F(\idd_\ga)\|=1$.

We denote by $\ga^\F:=\{a\in\ga\mid \F(a)=a\}$ the fixed-point subspace. It is closed, and satisfies $\big(\ga^\F\big)^*=\ga^\F$ since $\F$ is a $*$-map. If $\ga$ is unital, $\idd_\ga\in\ga^\F$, and therefore $\ga^\F$ is an operator system. If $\F$ is multiplicative, then $\ga^\F$ is a $*$-subalgebra of $\ga$.

We denote by $\cs(\ga)^\F:=\{\f\in\cs(\ga)\mid \f\circ\F=\f\}$ the set of all invariant states under the dynamics inherited by $\F$ on $\ga$. It is convex and weak-$*$ locally compact. 
If $\ga$ is unital, 
$\cs(\ga)^\F$ is weak-$*$ compact, and the extreme invariant states are called {\it ergodic}, see {\it e.g.} \cite{Sa}, Definition 3.1.9.
The set of ergodic states is denoted by $\partial\big(\cs(\ga)^\F\big)$. As usual, if $\om\in\partial\big(\cs(\ga)^\F\big)$, we say that the $C^*$-dynamical system $(\ga,\F,\om)$ is 
{\it ergodic}.

A $C^*$-dynamical system $(\ga,\F)$ is said to be {\it uniquely ergodic} if
\begin{itemize}
\item[(i)] $\cs(\ga)^\F=\emptyset$ when $\ga$ is not unital,
\item[(ii)]  $\cs(\ga)^\F=\{\f\}$, a {\it singleton}, when $\ga$ is unital.
\end{itemize}
We recall that for uniquely ergodic systems, in the unital case, one has $\frac1{n}\sum_{k=0}^{n-1}\F^k(a)\to\f(a)\idd_\ga$
in norm ({\it e.g.} \cite{AD}, Theorem 3.2). In particular, it also follows that
\begin{equation}
\label{xco}
\frac1{n}\sum_{k=0}^{n-1}\f(a\F^k(b))\xrightarrow[n \to +\infty]{}\f\big(a(\f(b)\idd_\ga)\big)=\f(a)\f(b)\,.
\end{equation}

Notice that, if $\ga$ is not unital, $\F$ extends in a unique way to $\ga_1:=\ga+\bc1$, the unital $C^*$-algebra obtained by adjoining the identity $1\equiv\idd_{\ga_1}$ to $\ga$. In such a way, 
$$
\cs(\ga_1)^\F=\cs(\ga)^\F\bigcup\{\om_\infty\}\,,
$$
the one-point compactification of $\cs(\ga)^\F$ with {\it the state at infinity} $\om_\infty$. Since $\om_\infty$ is automatically $\F$-invariant, the correct definition of unique ergodicity
for non-unital $C^*$-dynamical systems $(\ga,\F)$ must require the absence of any $\F$-invariant state in $\cs(\ga)^\F$. The reader is referred to \cite{F5}, Section 5, for some illustrative examples of such a situation. From now on we suppose that the involved $C^*$-algebra $\ga$ is unital with identity $\idd_\ga=:I$.
The $C^*$-dynamical system $(\ga,\F)$ is said to be {\it topologically ergodic} if $\ga^\F=\bc I$. 
We next single out the definition of a \emph{uniquely ergodic w.r.t. the fixed-point subalgebra} 
 $C^*$-dynamical system $(\ga,\a)$, where $\a$ is a $*$-automorphism or merely a multiplicative $*$-map.
In light of Theorem 3.2 in \cite{AD} and  Theorem 2.1 in \cite{FM3}, this definition can be formulated  by requiring any of the
following equivalent conditions.

\begin{defin}
\label{abdy}
The $C^*$-dynamical system $(\ga,\a)$ is said to be {\it uniquely ergodic with respect to the fixed-point subalgebra} ({\it cf.}  \cite{AD, FM3}) if one of the following equivalent properties holds true:
\begin{itemize}
\item[(i)] every bounded linear functional on $\ga^\a$ has a unique bounded, $\a$-invariant linear extension to $\ga$;
\item[(ii)] every state on $\ga^\a$ has a unique $\a$-invariant state-extension to $\ga$;
\vskip.1cm
\item[(iii)] $\overline{\ga^\a+\{a-\a(a)\mid a\in\ga\}}=\ga$;
\item[(iv)] the sequence $\Big(\frac1{n}\sum_{k=0}^{n-1}\a^k\Big)_n$ of the Ces\`{a}ro averages
converges pointwise in norm;
\item[(v)] the sequence $\Big(\frac1{n}\sum_{k=0}^{n-1}\a^k\Big)_n$ of the Ces\`{a}ro averages
converges pointwise in the weak topology;
\vskip.1cm
\item[(vi)] $\ga^\a+\overline{\{a-\a(a)\mid a\in\ga\}}=\ga$.
\end{itemize}
\end{defin}

In particular, $(\ga,\a)$ is uniquely ergodic w.r.t. the fixed-point subalgebra if and only if the ergodic averages $\frac1{n}\sum_{k=0}^{n-1}\a^k$ converge, pointwise in norm, 
to a conditional expectation onto $\ga^\a$ which is automatically invariant under $\a$.

For more general $C^*$-dynamical systems $(\ga,\F)$ based on completely positive unital maps $\F$, the fixed-point subspace $\ga^\F$ is an operator system and the ergodic averages $\frac1{n}\sum_{k=0}^{n-1}\F^k$ converge, pointwise in norm, to a completely positive
invariant projection onto $\ga^\F$. 

 We also note that:
\begin{itemize}
\item[(a)] if $\ga^\a=\bc I$, then the unique ergodicity w.r.t. the fixed-point subalgebra of $(\ga,\a)$ reduces to the usual unique ergodicity with  $\ce=\f(\,{\bf \cdot}\,)I$, where 
$\f\in\cs(\ga)$ is the unique $\a$-invariant state;
\item[(b)] if the $C^*$-dynamical system $(\ga,\a)$ is uniquely ergodic, then it is topologically ergodic.
\end{itemize}

When working with  a specific invariant state, we also denote a $C^*$-dynamical system by a triple $(\ga,\F,\f)$. For example, if 
$(\ga,\F)$ is uniquely ergodic with $\f$ as the unique invariant state, we directly write $(\ga,\F,\f)$.

For a $C^*$-dynamical system $(\ga,\F,\f)$, consider the GNS representation $\big(\ch_\f,\pi_\f,\xi_\f\big)$. Then there exists a unique contraction $V_{\f,\F}\in\cb(\ch_\f)$, which is unitary in the case of $*$-automorphisms or an isometry for multiplicative maps, such that $V_{\f,\F}\xi_\f=\xi_\f$, and
$$
V_{\f,\F}\pi_\f(a)\xi_\f=\pi_\f(\F(a))\xi_\f\,,\quad a\in\ga\,,
$$
see {\it e.g.} \cite{NSZ}, Lemma 2.1.
The quadruple $\big(\ch_\f,\pi_\f, V_{\f,\F},\xi_\f\big)$ is called the {\it covariant GNS representation} associated to $(\ga,\F,\f)$.
From now on, we deal only with $*$- automorphisms without any further mention by noticing that part of the forthcoming analysis can be extended to general dynamical systems based on completely positive maps.

For a given $C^*$-dynamical system $(\ga,\F)$, an invariant state $\f\in\cs(\ga)^\F$  is said to be {\it weakly clustering} ({\it e.g.} \cite{AKTH}, Definition II.1) if
$$
\lim_{n\to+\infty}\frac1{n}\sum_{k=0}^{n-1}\f(a\F^k(b))=\f(a)\f(b)\,,\quad a,b\in\ga\,.
$$
It is well known that, for $\f\in\ga^\F$,

\smallskip

\centerline{$\f$ is weakly clustering $\iff V_{\f,\F}$ is mean ergodic $\Longrightarrow \f$ is ergodic,}

\smallskip

\noindent
and the last implication becomes an equivalence if the subset 
\begin{equation}
\label{abset}
E_1^{V_{\f,\F}}\pi_\f(\ga)E_1^{V_{\f,\F}}\subset\cb(\ch_\f)
\end{equation}
is made of mutually commuting 
operators, see {\it e.g}. \cite{Sa}, Proposition 3.1.10.\footnote{When the group $\bz$ is acting, a state $\f\in\cs(\ga)^\F$ is said to be $\bz$-abelian if the set \eqref{abset}
is made of commuting operators. For general groups $G$, we speak of {\it $G$-abelian states}.}

We include the following fact, perhaps known to experts:
\begin{prop}
\label{scrz}
Let $\ga$ be a $C^*$-algebra, and $\f\in\cs(\ga)$. If the support of $\f$ in the bidual is central, then
$\ker(\pi_\f)=\gpn_\f$.
\end{prop}
\begin{proof}
Suppose $a\in\gpn_\f$, then $\|\pi_\f(a)\xi_\f\|=0$. Since $\xi_\f$ is separating for $\pi_\f(\ga)''$, then $\pi_\f(a)=0$, and therefore $a\in\ker(\pi_\f)$. The proof follows as the reverse inclusion ({\it cf.} \eqref{idek}) holds true.
\end{proof}
As an immediate consequence, we find the well-known fact that a state on a commutative $C^*$-algebra is faithful if and only if its GNS representation is faithful. 
This is by no means true in the noncommutative setting, as is easily realized by considering a pure (vector) state on the $C^*$-algebra $\mathcal{K}(\mathcal{H})$ of compact
operators on a Hilbert space $\mathcal{H}$.

Proposition \ref{scrz} suggests some natural definitions concerning uniquely ergodic $C^*$-dynamical systems. Indeed, let $(\ga,\a,\f)$ be an uniquely ergodic $C^*$-dynamical system,
\begin{itemize}
\item[(i)] we say that $(\ga,\a,\f)$ is {\it strictly ergodic} ({\it e.g.} \cite{Fu}) if $\f$ is faithful;
\item[(ii)] we say that $(\ga,\a,\f)$ is {\it sharply ergodic} if it is strictly ergodic, and in addition the support of $\f$ in the bidual is central.
\end{itemize}

Since the support of any state is automatically central in commutative cases, the above definitions do differ only when $\ga$ is not commutative.

We say that the $C^*$-dynamical system $(\ga,\a)$ is {\it minimal} if the unique $\a$-invariant hereditary $C^*$-subalgebras of $\ga$ are $\bc I$ and $\ga$ itself.
We easily get that 

\smallskip

\centerline{$(\ga,\a)$ is minimal $\Longrightarrow (\ga,\a)$ is topologically ergodic,}

\smallskip

\noindent
but the converse is false as  is known already from the classical theory, see {\it e.g.} \cite{F1}, Section 4.

\section{The $C^*$-crossed product}
\label{cfprx}

We fix a $C^*$-dynamical system $\big(\ga,\{\a_g\mid g\in G\}\big)$ made of a $C^*$-algebra and a pointwise norm-continuous representation $g\in G\mapsto\a_g\in\aut(\ga)$ of the locally compact group $G$. With any such dynamical system, it is possible to associate a new $C^*$-algebra $\ga \rtimes_\a G$ called $C^*$-{\it crossed product} of $\ga$ and the action $\a_g$ of the group $G$. Henceforth, we will be dealing with the simpler situation of the discrete group $\bz$,  which corresponds to having a single $*$-automorphism  along with its powers $\a\in\aut(\ga)$, $\a_n:=\a^n$, $n\in\bz$, see {\it e.g.} \cite{D}.

For a given $C^*$-dynamical system $(\ga,\a)$ with $\a\in\aut(\ga)$, consider all formal finite sums
$$
\bigg\{\sum_{k\in F}V^ka_k\mid a_k\in\ga,\,\,F\subset\bz\,\,\text{finite set}\,\bigg\}\,,
$$
where $V$ is an outer unitary satisfying the commutation relation
$$
VaV^*=\a(a)\,,\quad a\in\ga\,.
$$
As is known, $\mathfrak{A} \rtimes_{\alpha}\mathbb{Z}$ is nothing but the maximal $C^*$-completion of $\gb_o$, see {\it e.g.} \cite{D}.
The crossed product $\mathfrak{A}\rtimes_\alpha \mathbb{Z}$ is acted upon 
the one-dimensional torus $\mathbb{T}$ through the so-called {\it gauge automorphisms}, which are uniquely determined by their action on the generators:
$$
\r_z(a):=a\,,\,\,\,\r_z(V):=zV\,,\quad a\in\ga\,,\,\,z\in\bt\,.
$$
This allows one to produce 
a conditional expectation $E$ of $\mathfrak{A} \rtimes_{\alpha}\mathbb{Z}$ onto $\mathfrak{A}$ by averaging the action of $\mathbb{T}$, {\it i.e.}
\begin{equation}
\label{gcxe}
E(x):=\oint \rho_z(x)\frac{\di z}{2\pi\imath z}\,,\quad x\in\mathfrak{B}\,.
\end{equation}

Let $\om_o\in\cs(\ga)$. Composing $\omega_o$ after $E$ yields a state 
\begin{equation}
\label{stainv}
\omega:=\omega_o \circ E
\end{equation}
on $\mathfrak{A} \rtimes_{\alpha}\mathbb{Z}$. Notice that, since $E$ is faithful, $\om$ will be faithful provided 
$\om_o$ is so. 
In addition,
\begin{equation}
\label{tga}
\om(V^na)=\om_o(a)\delta_{n,0}=\om(aV^n)\,,\quad a\in\ga\,,\,\,n\in\bz\,.
\end{equation}
as is easily checked.

At this point, we would like to point out that the centrality the support $s(\om_o)\in\ga^{**}$ of the state $\omega_o$ carries over to that $s(\om)\in(\mathfrak{A} \rtimes_{\alpha}\mathbb{Z})^{**}$ of that of
$\omega$ at the level of the crossed product $\mathfrak{A}\rtimes_\alpha\mathbb{Z}$. We assume $\om_o\in\cs(\ga)^\a$, and note that we do not use such a property but in Corollary \ref{inczfx}. Thus we have chosen not to pursue the study of the general case when $\om_o$ is not invariant under $\a$.

In such a situation, this property easily follows by \cite{T00}, Corollary 5.13. We sketch the proof for the convenience of the reader.
\begin{prop}
\label{czni}
If $s(\omega)\in Z\big((\mathfrak{A}\rtimes_\alpha\mathbb{Z})^{**}\big)$, then  $s(\omega_o)\in Z(\mathfrak{A}^{**})$, and the reverse implication holds true with the additional condition $\om_o\in\cs(\ga)^\a$.
\end{prop}
\begin{proof}
We note that, if the support of $\om$ is central, then the support of $\om_o$ is of course central as well.
Then we only need to prove the reverse implication under the additional condition of the invariance of $\om_o$ under $\a$. To this aim,
we will show that $\xi_\omega$ is a cyclic vector for $\pi_\omega(\mathfrak{A}\rtimes_\alpha\mathbb{Z})'$.

On $\ch_\om:=\ell^2(\bz;\ch_{\om_o})$, define for $k\in\bz$,
\begin{align*}
&(\pi_\om(a)\xi)(k):=\pi_{\om_o}(\a^{-k}(a))\xi(k)\,,\quad a\in\ga\,,\\
&(\pi_\om(V^h)\xi)(k):=\xi(k-h)\,,\quad h\in\bz\,.
\end{align*}
Since
$$
\pi_\om(V^h)\pi_\om(a)\pi_\om(V^{-h})=\pi_\om(\a^h(a))\,,\quad a\in\ga\,,\,\,h\in\bz\,,
$$
by the universal property of the $C^*$-crossed product $\pi_\om$ extends to a representation of $\mathfrak{A}\rtimes_\alpha\mathbb{Z}$. With $\xi_\om(k)=\xi_{\om_o}\d_{k,0}$
$k\in\bz$, it is matter of routine to check that $(\ch_\om, \pi_\om,\xi_\om)$ provides the GNS representation of $\om$, up to unitary equivalence.

Suppose now that $\om_o$ is invariant under $\a$ with covariant GNS representation $(\ch_{\om_o}, \pi_{\om_o}, V_{\om_o,\a},\xi_{\om_o})$, then $\pi_\om$ comes from the von Neumann crossed product construction, see {\it e.g.} \cite{BR}, Definition 2.7.3. Therefore,
the commutant $\pi_\om(\mathfrak{A}\rtimes_\alpha\mathbb{Z})'$ is generated by
\begin{align*}
&(\pi'(B)\xi)(k):=B\xi(k)\,,\quad B\in\pi_{\om_o}(\ga)'\,,\\
&(\l'(h)\xi)(k):=V^{h}_{\om_o,\a}\xi(k+h)\,,\quad h\in\bz\,.
\end{align*}

Consider an element $X\in\pi_\om(\mathfrak{A}\rtimes_\alpha\mathbb{Z})'$ of the form
$$
X=\sum_{h\in\bz}\l'(h)\pi'(B_h)\,,\quad B_h\in\pi_{\om_o}(\ga)'\,,\,\,\,h\in\bz\,,
$$
with $B_h=0$ but finitely many of them. We obtain 
$$
(X\xi_\om)(k)=V^{-k}_{\om_o,\a}B_{-k}\xi_{\om_o}=\ad{}\!_{V^{-k}_{\om_o,\a}}(B_{-k})\xi_{\om_o}\,.
$$

By taking into account that $\ad{}\!_{V^{k}_{\om_o,\a}}(\pi_{\om_o}(\ga)')=\pi_{\om_o}(\ga)'$, we deduce that $\xi_{\om_o}$ is cyclic for $\pi_{\om_o}(\ga)'$ if and only if $\xi_\om$ is cyclic for $\pi_\om(\mathfrak{A}\rtimes_\alpha\mathbb{Z})'$.
\end{proof}
We end the present section by recalling the following result in \cite{D}, Theorem VIII.2.2, see also \cite{ACR} for some new general results in the noncommutative setting on the topic. Indeed, with $E:\ga\rtimes_\a\bz\rightarrow\ga$ given in \eqref{gcxe},
$x\in\mathfrak{A}\rtimes_\alpha\mathbb{Z}$ admits the \lq\lq Fourier expansion"
\begin{equation}
\label{cesexp}
x=\sum_{n\in\bz}V^nE(V^{-n}x)\,,
\end{equation}
where the convergence of the above sum is meant in norm, in the sense of Ces\`aro.

Considering the natural inclusions of von Neumann algebras $\pi_\om(\ga)''\subset\pi_\om(\ga\rtimes_\a\bz)''$ associated to the inclusions of $C^*$-algebras
$\ga\subset\ga\rtimes_\a\bz$, we can easily get the extension $\widetilde{E}:\pi_\om(\ga\rtimes_\a\bz)''\rightarrow\pi_\om(\ga)''$ of $E$, obtained by averaging the gauge action  
$z\rightarrow \tilde \r_z$, the extension of $z\rightarrow \r_z$
to the von Neumann algebra
$\pi_\omega(\mathfrak{A}\rtimes_\alpha\mathbb{Z})''$,
which provides a normal conditional expectation of $\pi_\om(\ga\rtimes_\a\bz)''$ onto $\pi_\om(\ga)''$.

We show that the \lq\lq Fourier expansion" for elements of $\pi_\om(\ga\rtimes_\a\bz)''$, probably known to the experts and analogous to that in \eqref{cesexp}, holds true as well.
\begin{prop}
\label{ftrf}
Each $X\in\pi_\om(\ga\rtimes_\a\bz)''$ admits the expansion
\begin{equation}
\label{cesexp1}
X=\sum_{n\in\bz}\pi_\om(V^n)\widetilde{E}(\pi_\om(V^{-n})X)\,,
\end{equation}
where the convergence of the above sum is meant in the strong operator topology, in the sense of Ces\`aro.\footnote{The last assertion means nothing else that, for each 
$\xi\in\ch_\om$, the series
$\sum_{n\in\bz}\pi_\om(V^n)\widetilde{E}(\pi_\om(V^{-n})X)\xi$ converges to $X\xi$ in the sense of Ces\`aro in the Banach space $\ch_\om$.}
\end{prop}
\begin{proof}
We start by noticing that $\bt\ni z\rightarrow\tilde \r_z\in\aut(\pi_\om(\mathfrak{A}\rtimes_{\alpha}\bz)'')$ is pointwise continuous in the strong operator topology. With $\tilde \r_z=\ad_{U_z}$, $U_z$ implementing the gauge automorphism $\tilde \r_z$ on $\pi_\om(\mathfrak{A}\rtimes_{\alpha}\bz)''$, this follows by noticing that $z\rightarrow U_z$ is continuous in the strong operator topology, and the latter is jointly continuous on the bounded parts of $\cb(\ch_\om)$.

For $l\in\bn$, let $s_l(X)=\sum_{|k|\leq l}\pi_\om(V^k)\widetilde{E}(\pi_\om(V^{-k})X)$ the partial sum of \eqref{cesexp1}, and $\s_n(X)=\frac{1}{n}\sum_{l=0}^{n-1} s_l(X)$, after some standard computations, the following equality is got to
$$
\s_n(X)=\frac{1}{2\pi}\int_{-\pi}^{\pi} F_n(\theta) \tilde{\rho}_{e^{\imath\theta}}(X){\rm d}\theta\,,
$$
where $F_n(\theta):=\frac{1}{n}\frac{\sin^2(\frac{n\theta}{2})}{\sin^2(\frac{\theta}{2})}$ is nothing
but the Fej\'er kernel. 

We recall that $\frac{1}{2\pi}\int_{-\pi}^\pi F_n(\theta)d\theta=1$, and for any
$\varepsilon,\delta>0$ there exists $n_{\varepsilon,\delta}>0$ such that $\frac{1}{2\pi}\int_{-\delta}^\delta F_n(\theta)d\theta\geq 1-\varepsilon$ for any $n>n_{\epsilon, \delta}$. 

Fix $\xi\in\ch_\om$, and $\varepsilon>0$. Since $\tilde \r_z$ is pointwise continuous in the strong operator topology,
there exists $\delta_\eps>0$ such that $|\theta|<\delta_\eps$ implies 
$\|\tilde{\rho}_{e^{\imath\theta}}(X)\xi-X\xi\| <\varepsilon$. Collecting together, we have
\begin{align*}
&\|\s_n(x)\xi-X\xi\|\leq \frac{1}{2\pi}\int_{-\pi}^\pi F_n(\theta)\|\tilde{\rho}_{e^{\imath\theta}}(X)\xi-X\xi\| {\rm d}\theta\\
=&\frac{1}{2\pi}\int_{-\delta_\eps}^{\delta_\eps}\!\!\!\!F_n(\theta)\|\tilde{\rho}_{e^{\imath\theta}}(X)\xi-X\xi\| {\rm d}\theta
+\frac{1}{2\pi}
\int_{|\theta|>\delta_\eps}\!\!\!\!F_n(\theta)\|\tilde{\rho}_{e^{\imath\theta}}(X)\xi-X\xi\| {\rm d}\theta\\
\leq&\varepsilon+ 2\|X\|\|x\|\varepsilon\,,
\end{align*}
as long as $n>n_{\varepsilon, \delta_\eps}$.
\end{proof}
We would like to note that the above inversion formulae can be generalised to other means reproduced by any approximate identity. We report the case associated to the Abel
mean, $0<r<1$, for which the analogous of \eqref{cesexp} and \eqref{cesexp1} are
\begin{align*}
x=&\lim_{r\uparrow1}\sum_{n\in\bz}r^{|n|}V^nE(V^{-n}x)\,,\quad x\in \ga\rtimes_\a\bz\,,\\
X=&\lim_{r\uparrow1}\sum_{n\in\bz}r^{|n|}\pi_\om(V^n)\widetilde{E}(\pi_\om(V^{-n})X)\,,\quad X\in\pi_\om(\ga\rtimes_\a\bz)''\,,
\end{align*}
where the convergence is in norm and in the strong operator topology, respectively.\footnote{Concerning the case of $X\in\pi_\om(\ga\rtimes_\a\bz)''$, we easily have the convergence in the strong-* operator topology as well.}

An immediate consequence of the last results is an inversion formula for the Fourier series of functions in $L^\infty(\bt,\di\l)$. 
Indeed, consider a countable dense set $\{h_n\mid n=1,2\dots\}$ of the unit ball of $L^2(\bt,\di\l)$. On the unit ball $L^\infty(\bt,\di\l)_1$ of $L^\infty(\bt,\di\l)$, define a distance given by
$$
d(f,g):=\sum_{n=1}^{+\infty}\frac1{2^n}\sqrt{\oint\frac{\di z}{2\pi\imath z}|(f(z)-g(z))h_n(z)|^2}\,,
$$
whose induced topology is nothing else than the strong operator topology, or equivalently the strong-* operator topology, restricted to $L^\infty(\bt,\di\l)_1$, see {\it e.g.} \cite{T}, Proposition II.2.7.

Proposition \ref{ftrf} leads immediately to
$$
\lim_nd(f,\s_n(f))=0\,,\quad f\in L^\infty(\bt,\di\l)_1\,,
$$
where
$$
\s_n(f):=\frac1{n}\sum_{k=0}^{n-1}\Big(\sum_{|l|\leq k}\hat f(l)z^l\Big)\,.
$$

We also consider the Abel mean of $f\in L^\infty(\bt,\di\l)$ defined for $0<r<1$ as
$$
A_r(f):=\sum_{k\in\bz}\hat f(l)r^{|k|}z^k\,,
$$
where the above series is totally convergent in all segments $[a,b]$, $0<a,b<1$. Analogously, we get
$$
\lim_{r\uparrow1}d(f,A_r(f))=0\,,\quad f\in L^\infty(\bt,\di\l)_1\,.
$$

\section{On the canonical cone associated to standard von Neumann algebras}
\label{thecone}

Let $(\cam,\ch)$ be a von Neumann algebra equipped with a cyclic and separating vector $\xi\in\ch$, that is a standard vector. A {\it standard von Neumann algebra} will be then a triplet $(\cam,\ch,\xi)$ made of a von Neumann algebra $(\cam,\ch)$, and unit vector $\xi\in\ch$ which is standard for $\cam$.

For the standard von Neumann algebra $(\cam,\ch,\xi)$, since $\xi$ is separating, for $a\in\cam$, $a\xi\in\ch\mapsto a^*\xi$ uniquely defines a densely defined, closable anti-linear map 
$$
S_o(a\xi)=a^*\xi\,,\quad a\in\cam\,,
$$
whose closure, denoted by $S$, has polar decomposition $S=J\D^{1/2}$. The operators $S$, $\D$ and $J$ are called the {\it Tomita involution, modular operator and conjugation}, respectively. The main results of Tomita-Takesaki theory are the remarkable fact that $\ad_{\D^{\imath t}}$ leaves $\cam$ globally invariant and therefore defines a one-parameter group of $*$-automorphisms of $\cam$ called {\it the modular group}, and that $\cam$ acts in standard form on $\ch$ with conjugation $J$:
\begin{itemize}
\item[(i)] $J\cam J=\cam'$,
\item[(ii)] $JzJ=z^*$, $z\in Z(\cam)$.
\end{itemize}

A canonical self-dual cone $\cp\subset\ch$ ({\it e.g.} Haa) can be defined as follows:
$$
\cp:=\overline{\{JaJa\xi\mid a\in\cam\}}=\cp=\overline{\{\D^{1/4}a\xi\mid a\in\cam_+\}}\,.
$$

For the standard von Neumann algebra $(\cam,\ch,\xi)$, we include the following result without proof ({\it cf.} \cite{SZ}, E.10.20) relative to the {\it polar decomposition} of any vector 
$\eta\in\ch$.
\begin{prop}
\label{szcon}
Under the above notations, for each $\eta\in\ch$ we have $\eta=v_\eta|\eta|$, uniquely determined by the conditions: $|\eta|\in\cp$, $v_\eta\in\cam$ and $v_\eta^*v_\eta=p_{|\eta|}$, the self-adjoint projection onto the closed subspace $[\cam'|\eta|]$.
\end{prop}

We collect some results, which will come in useful in the next sections.
\begin{lem}
\label{cif0}
For $\eta\in\ch$ and $v\in\cam$ an isometry, we have $|v\eta|=|\eta|$.
\end{lem}
\begin{proof}
We have $v\eta=vv_\eta|\eta|$, and $(vv_\eta)^*(vv_\eta)=v_\eta^*v^*vv_\eta=v_\eta^*v_\eta=p_{|\eta|}$ because $v$ is an isometry. The assertion now follows by Proposition \ref{szcon}.
\end{proof}

For the remaining part of the section, we consider $u\in\cu(\ch)$ such that $u\cam u^*=\cam$ and $u\xi=\xi$.
\begin{lem}
\label{cif}
We have $u\cp=\cp$.
\end{lem}
\begin{proof}
By \cite{NSZ}, Lemma 3.3, $[u,J]=0$. Therefore,
\begin{align*}
u\cp\equiv u\overline{\{JaJa\xi\mid a\in\cam\}}
=&\overline{\{J(uau^*)J(uau^*)\xi\mid a\in\cam\}}\\
=&\overline{\{JaJa\xi\mid a\in\cam\}}=\cp\,.
\end{align*}
\end{proof}
\begin{lem}
\label{cif1}
For $\eta\in\ch$, we have $u|\eta|=|u\eta|$.
\end{lem}
\begin{proof}
We have $u\eta=uv_\eta|\eta|=(uv_\eta u^*)(u|\eta|)$. Since $u|\eta|\in\cp$ by Lemma \ref{cif}, thanks to Proposition \ref{szcon} it is enough to show that $\cam\ni uv_\eta^*v_\eta u^*=p_{u|\eta|}$.

We first get
$$
u[\cam'|\eta|]=[u(\cam')u^*(u|\eta|)]=[\cam'(u|\eta|)]\,,
$$
and thus $u[\cam'|\eta|]^\perp=[\cam'(u|\eta|)]^\perp$ because $u$ is unitary. 

Consequently,
$$
uv_\eta^*v_\eta u^*[\cam'(u|\eta|)]=uv_\eta^*v_\eta[\cam'|\eta|]=u[\cam'|\eta|]=[\cam'(u|\eta|)]\,,
$$
which means $uv_\eta^*v_\eta u^*\geq p_{u|\eta|}$. Analogously, $uv_\eta^*v_\eta u^*[\cam'(u|\eta|)]^\perp=[\cam'(u|\eta|)]^\perp$, which leads to $uv_\eta^*v_\eta u^*=p_{u|\eta|}$.
\end{proof}
\begin{lem}
\label{cif2}
If $u$ is mean-ergodic, then $\cp\bigcap E_1^u\ch=\overline{\br_+}\xi$.
\end{lem}
\begin{proof}
Since $\eta\in E_1^u\ch\iff\eta=z\xi$ for some $z\in\bc$, the unicity of the polar decomposition of $(e^{\imath{\rm arg}(z)}\idd_H)(|z|\xi)=\eta=v_\eta|\eta|$ leads to $|\eta|=|z|\xi$ and $v_\eta=e^{\imath{\rm arg}(z)}\idd_H$.
\end{proof}

\section{The noncommutative skew-product}
\label{anspr}

The aim of the present section is to provide the quantum generalisation of the Anzai skew-product, first appeared in \cite{A}, see also \cite{OP, DFGR}, for the quantum case on the noncommutative torus.

We start with a $C^*$-dynamical system $(\mathfrak{A}, \theta)$ made of the unital  $C^*$-algebra $\ga$, the automorphism 
$\th\in\aut(\ga)$ of $\mathfrak{A}$. We next consider a second automorphism $\alpha\in\aut(\ga)$.

To this end, we start by extending the action of $\theta$ on $\ga$ to an automorphism acting on the crossed product
$\mathfrak{A} \rtimes_{\alpha}\mathbb{Z}$, which is well known when $\theta$ and $\alpha$ commute, see {\it e.g.} \cite{CH}, Section 2.1. The construction  is also
possible in more generality, for instance  when $\theta$ and $\alpha$ commute up to a unitary
in the $C^*$-tensor category ${\rm End}(\mathfrak{A})$.
More explicitly, we assume that there exists $u\in\mathcal{U}(\ga)$ such that
\begin{equation}
\label{inter}
u(\a\circ\th)(a)=(\th\circ\a)(a)u\,,\quad a\in\ga\,.
\end{equation}

Now, associated with any intertwining unitary $u$ as above,
it is possible to define an automorphism $\Phi_{\theta, u}\in\rm{Aut}(\mathfrak{A} \rtimes_{\alpha}\mathbb{Z})$, whose action on the generators is 
\begin{equation}
\label{han}
\begin{split}
&\F_{\th,u}(a):=\th(a)\,,\quad a\in\ga\,,\\
&\F_{\th,u}(V):=uV\,.
\end{split}
\end{equation}

For the convenience of the reader, we include a proof that the above formulae do in fact yield a well-defined automorphism
of $\mathfrak{A} \rtimes_{\alpha}\mathbb{Z}$. Before doing that, we also note that 
$$
\F_{\th,u}(V^*)=V^*u^*=\a^{-1}(u^*)V^*\,.
$$
\begin{prop}
The map $\F_{\th,u}$ in \eqref{han} uniquely extends to a unital $*$-automorphism denoted, with a slight abuse of notation, again with $\F_{\th,u}$, whose inverse is
$\F_{\th,u}^{-1}=\F_{\th^{-1},\th^{-1}(u^*)}$.
\end{prop}
\begin{proof}
Thanks to the universal property enjoyed by crossed products and the definition for $\F_{\th,u}(V^*)$ in \eqref{han}, in order to obtain a unital $*$-endomorphism, for $a\in\ga$ we should satisfy
$$
\F_{\th,u}(V)\F_{\th,u}(a)\equiv\F_{\th,u}(Va)=\F_{\th,u}(\a(a)V)\equiv\F_{\th,u}(\a(a))\F_{\th,u}(V)\,.
$$
Therefore, we have to check $\F_{\th,u}(V)\F_{\th,u}(a)=\F_{\th,u}(\a(a))\F_{\th,u}(V)$, $a\in\ga$.

Indeed, by the assumption \eqref{inter},
\begin{align*}
\F_{\th,u}(V)&\F_{\th,u}(a)=uV\th(x)=u(\a\circ\th)(a)V\\
=&(\th\circ\a)(a)uV=\F_{\th,u}(\a(a))\F_{\th,u}(V)\,.
\end{align*}

We now notice that $\th^{-1}$ satisfies the same condition \eqref{inter} w.r.t. the unitary $\th^{-1}(u^*)$:
$$
\th^{-1}(u^*)(\a\circ\th^{-1})(a)=(\th^{-1}\circ\a)(a)\th^{-1}(u^*)\,,\quad a\in\ga\,,
$$
and therefore $\F_{\th^{-1},\th^{-1}(u^*)}$, defined as in \eqref{han}, provides another unital $*$-endomorphism of 
$\ga\rtimes_\a\bz$ which realises a left and right inverse of $\F_{\th,u}$.
\end{proof}
To ease the computations we will make in the sequel, we need to re-express powers $(uV)^n$, $n\in\mathbb{Z}$, in a more
convenient way, which is done below. 
\begin{lem}
\label{coo0}
With the notation established above, the equality $(uV)^n=u_nV^n$ holds  for every $n\in\mathbb{Z}$, where the double sequence
$\{u_n\mid n\in\mathbb{Z}\}\subset\mathfrak{A}$ is given by
\begin{equation}
\label{coo}
u_{n}=\left\{\begin{array}{ll}
                     I\,,\,\,\, n=0\,,&\\[1ex]
                      u\alpha(u)\ldots \alpha^{n-1}(u)\,,   \,\,\,
                      n>1\,,& \\[1ex]
                   \a^{-1}(u^*)\a^{-2}(u^*)\ldots\a^{n}(u^*)\,,  \,\,\,
                      n<0\,.&
                    \end{array}
                    \right.
\end{equation}
\end{lem}
\begin{proof}
It is a straightforward induction, which can be safely left out.
\end{proof}
Analogously, we can write
\begin{equation}
\label{coo1}
(uV)^n=V^n\a^{-n}(u_n)\,,\quad n\in\bz\,.
\end{equation}

Suppose now that a state $\om_o$ is invariant under the action of $\th$, that is $\om_o\in\cs(\ga)^\th$. Therefore, we have a $C^*$-dynamical system $(\ga,\th,\om_o)$. Correspondingly, 
we want to exhibit the {\it skew-product dynamical system} $(\mathfrak{A} \rtimes_{\alpha}\mathbb{Z},\F_{\th,u},\om)$ where,  
$\om\in\cs(\mathfrak{A} \rtimes_{\alpha}\mathbb{Z})$ is simply defined by \eqref{stainv} as
$\om:=\om_o\circ E$.

Concerning the invariance of $\om$ w.r.t. $\F_{\th,u}$, we have
\begin{prop}
With the notations set above, the state $\omega\in \mathcal{S}(\mathfrak{A} \rtimes_{\alpha}\mathbb{Z})$ is $\F_{\th,u}$-invariant, i.e.
$\omega\circ\F_{\th,u}=\omega$.
\end{prop}
\begin{proof}
Notice that it is enough to check the equality on the generators of the form $aV^n$, $a\in\mathfrak{A}$ and $n\in\mathbb{Z}$.
On the one hand, by \eqref{tga} we have $\om(aV^n)=\om_o(a)\delta_{n,0}$. 
By Lemma \ref{coo0} and the invariance of $\om_o$ under $\th$, we also have 
\begin{align*}
\omega(\Phi_{\th,u}(aV^n))=&\omega(\Phi_{\th,u}(a)\Phi_{\th,u}(V^n))=\omega(\theta(a)(uV)^n)\\
=&\omega(\theta(a)u_nV^n)=\omega_o(\theta(a)u_n)\delta_{n,0}\\
=&\omega_o(\theta(a))\delta_{n,0}=\omega_o(a)\delta_{n,0}\,.
\end{align*}

Collecting everything together, we have
$$
\omega(\Phi_{\th,u}(aV^n))=\omega_o(a)\delta_{n,0}=\om(aV^n)\,,
$$
which ends the proof.
\end{proof}
To sum up, starting with a $C^*$-dynamical system $(\ga,\th,\om_o)$ and $\a\in\aut(\ga)$, such that
$\theta\circ\alpha$ and $\alpha\circ\theta$ are unitarily equivalent in ${\rm End}(\mathfrak{A})$ by the unitary $u\in\cu(\ga)$, we can define a  
skew-product dynamical system associated to $u$, extending the Anzai skew-product constructed in the classical case, and therefore providing the $C^*$-dynamical system 
$(\mathfrak{A} \rtimes_{\alpha}\mathbb{Z},\Phi_{\th,u},\om)$ for the associated  $C^*$-crossed product.

\section{A comparison with alpha-induction}
\label{aind}

The extensions considered above have already appeared in the literature with the name \emph{alpha-induction}, even if in some very particular cases arising from the theory of factor-subfactor inclusions. Alpha-induction is designed to produce extensions of endomorphisms in the context of quantum field theory or, more generally, subfactors as long as they belong to an appropriate braided tensor category. 
 
Let $(\mathcal{C},\otimes,\id)$ be a strictly associative $C^*$ tensor category. It is said to be {\it braided} if it is equipped with a unitary braiding $\varepsilon^+$, that is a family 
$$
\{\varepsilon^+_{X,Y} \in\text{Hom}(X\otimes Y,Y \otimes X)\mid X,Y \in \ob(\cc)\}
$$
of unitary isomorphisms
satisfying
\begin{itemize}
\item[(i)] {\it naturality}: for any $X,X^\prime,Y,Y^\prime\in\ob(\mathcal{C})$, $f\in \text{Hom}(X,X^\prime)$ and $g \in \text{Hom}(Y,Y^\prime)$,
\begin{equation}
\label{naturality}
(g\otimes f)\circ \varepsilon^+_{X,Y} =\varepsilon^+_{X^\prime,Y^\prime} \circ (f\otimes g)\,;
\end{equation}
\item[(ii)] {\it tensoriality}: for every $X,Y,Z\in\ob(\mathcal{C})$,
\begin{equation}
\label{tensoriality2}
\begin{split}
&\varepsilon^+_{X\otimes Y,Z} = (\varepsilon^+_{X,Z} \otimes \id_{Y} )  \circ (\id_{X} \otimes \varepsilon^+_{Y,Z})\,,\\
&\varepsilon^+_{X,Y \otimes Z} = (\id_Y \otimes \varepsilon^+_{X,Z}) \circ (\varepsilon^+_{X,Y} \otimes \id_Z)\,.
\end{split}
\end{equation}
\end{itemize}

Note that, if $\varepsilon^+$ is a unitary braiding for $(\mathcal{C},\otimes,\id)$, then also 
$$
\{\varepsilon^-_{X,Y}:=(\varepsilon^{+}_{Y,X})^*\mid X,Y \in\ob(\cc)\}
$$
is a unitary braiding for $(\mathcal{C},\otimes,\id)$.\\

In the rest of the present section, we assume that 
$$
I:=\idd_{\ch}\in\mathcal{N}\subset\mathcal{M}\subset\cb(\ch)
$$ 
is a factor-subfactor inclusion of von Neumann algebra acting on the separable Hilbert space $\ch$, with $\mathcal{N}$ and $\mathcal{M}$ of type $\ty{III}$.

Let $\xi\in\mathcal{H}$ be a standard vector for both $\mathcal{M}$ and $\mathcal{N}$ which exists by the Dixmier-Mar\'{e}chal Theorem ({\it cf.} \cite{DM}), $J_\xi^\mathcal{N}$ and $J_\xi^\mathcal{M}$ being the modular conjugations w.r.t. $\xi$ of $\mathcal{N}$ and $\mathcal{M}$, respectively.

The {\it canonical endomorphism} ({\it e.g.} \cite{LR}) $\gamma_\xi\in\ed(\mathcal{M})$ is defined for the subfactor $\mathcal{N}\subset\mathcal{M}$ as
$$
\gamma_\xi(m):=\ad\!{}_{(J_\xi^{\mathcal{N}}J_\xi^{\mathcal{M}})}(m)\,,\quad m\in\cam\,.
$$
If $\eta\in\mathcal{H}$ is another standard vector for both $\mathcal{N}$ and $\mathcal{M}$, there exists a unitary $u\in\cn$ ({\it cf.} \cite{LR}, Section 2)
such that 
$$
\gamma_\eta=\ad\!{}_{u}\circ\gamma_\xi\,.
$$

Since the reference vector $\xi\in\ch$ is fixed once for all, for the canonical endomorphism we write
$\gamma:=\gamma_\xi$.

Let $\hat\g\in\ed(\mathcal{N})$ be the restriction of $\gamma$ to $\mathcal{N}$. It is referred to as the dual canonical endomorphism, and provides a canonical endomorphism associated to the inclusion $\g(M)\subset N$.

Suppose that there is a tensor subcategory $\mathcal{C}$ of $\ed(\mathcal{N})$ which admits unitary braidings $\varepsilon^{\pm}$, is full, and contains as objects  the canonical endomorphism
$\hat\g$ and its sub-objects. In this situation, we say that the subfactor $\mathcal{N}\subset\mathcal{M}$ is \emph{braided}.
Then, for every $\sigma\in{\rm Obj}(\mathcal{C})$, we can consider its extensions to $\ed(\mathcal{M})$ via {\it alpha-induction} ({\it cf.} \cite{LR}, Proposition 3.9):
\begin{equation}
\label{alphainduc}
\alpha^{\pm}_\sigma:=\gamma^{-1}\circ\ad{}_{\varepsilon^\pm_{\sigma,\rho}}\circ \sigma\circ\gamma.
\end{equation}

In \cite{LR}, \eqref{alphainduc} is understood in the context of local nets, where $\sigma$ is a localizable, transportable endomorphism of the local net, and 
$\alpha^{\pm}_\sigma$ are endomorphisms of the quasi-local $C^*$-algebra of the local net. 

If $\mathcal{N}\subset\mathcal{M}$ has finite Jones index, in the  context of braided subfactors outlined above, 
\eqref{alphainduc} is known to be well defined and to yield $\alpha^{\pm}_\sigma\in\ed(\mathcal{M})$.

Since we are here interested in crossed products by the action of $\mathbb{Z}$, we show for completeness that alpha-induction from \eqref{alphainduc} is actually well defined and yields an element of $\ed(\mathcal{M})$, under the hypothesis that $\mathcal{N}\subset\mathcal{M}$ is an irreducible ({\it i.e.} $\mathcal{N}^\prime\bigcap\mathcal{M}=\mathbb{C}$), discrete, and braided inclusion. For such a purpose, we first recall the relevant definitions.

Suppose that there exists a normal faithful conditional expectation $E:\cam\to\cn$ of $\cam$ onto $\cn$, and consider the Jones projection $e_\mathcal{N}$ for 
$\mathcal{N}\subset \mathcal{M}$ associated to $E$.
Namely, fix the faithful normal state  $\omega_\xi\circ E:=\langle E(\,{\bf\cdot}\,)\xi,\xi\rangle$, and consider the standard vector $\Om\in\cp_\xi$ for 
$\mathcal{M}$ which induces $\omega_\xi\circ E$. The Jones projection $e_\mathcal{N}$ is defined as the self-adjoint
projection onto $\overline{\mathcal{N}\Omega}\subset\mathcal{H}$. 

A \emph{Pimsner-Popa basis} for $\mathcal{N}\subset \mathcal{M}$ w.r.t. $E$ is a countably family of elements $\{m_i\} \subset \mathcal{M}$, such that:
\begin{itemize}
\item[$(i)$] $m_i^* e_\mathcal{N} m_i$ are mutually orthogonal projections;
\item[$(ii)$] $\sum_i m_i^* e_\mathcal{N} m_i = I$, where the sum converges in the strong operator topology.
\end{itemize}
\begin{defin}
An irreducible subfactor $\mathcal{N}\subset\mathcal{M}$ of type $\ty{III}$ is said to be discrete if it has a (unique) normal faithful conditional expectation $E:\mathcal{M}\rightarrow\mathcal{N}$, and the dual canonical endomorphism $\hat\g$ is a countable direct sum of finite-index endomorphisms ({\it i.e.} $[\r_i(\cn):\cn]<+\infty$)
$$
\hat\g=\sum_i v_i\rho_i(\,{\bf\cdot}\,)v_i^*\,,
$$
where $v_i\in\mathcal{N}$ mutually orthogonal isometries satisfying $\sum_i v_i v_i^*=\id_\ch$, and the convergence is understood in the strong operator topology.
\end{defin}

\begin{prop}[{\it cf.} {\cite{DG}}]\label{prop:discrete} 
Let $\mathcal{N}\subset \mathcal{M}$ be an irreducible inclusion of type $\ty{III}$ factors. Then $\mathcal{N}\subset\mathcal{M}$ is discrete if and only if, there exists a, necessarily unique, normal, faithful conditional expectation $E:\mathcal{M}\rightarrow\mathcal{N}$, and a Pimsner-Popa basis $\{m_i\}_{i\in I}\subset\mathcal{M}$ w.r.t. $E$ whose elements 
fulfill 
$$
m_i n = \hat\g(n) m_i\,,\quad n\in \mathcal{N}\,,\,\,i\in I\,.
$$
\end{prop}
In order to check that \eqref{alphainduc} yields a well-defined $\alpha_\sigma^{\pm}\in\ed(\mathcal{M})$ in the context of the present paper, we report the following
\begin{lem}
\label{welldefinedalpha}
Let $\mathcal{N}\subset \mathcal{M}$ be an irreducible inclusion of type $\ty{III}$ factors.
Suppose that there is a braided tensor subcategory $\mathcal{C}$ of $\ed(\mathcal{N})$ which is full and contains $\hat\g$.
Then for any $\sigma\in\ob(\mathcal{C})$,
\begin{equation*}
\ad\!{}_{\varepsilon^\pm_{\sigma,\hat\g}}\circ \sigma\circ\gamma(\mathcal{M})\subset \gamma(\mathcal{M}).
\end{equation*}
\end{lem}
\begin{proof}
For $n\in\mathcal{N}$, we have $\ad\!{}_{\varepsilon^\pm_{\sigma,\hat\g}}\circ \sigma\circ\gamma(n)=\hat\g\circ\sigma(n)$.
We must check that
$$
\ad\!{}_{\varepsilon^\pm_{\sigma,\hat\g}}\circ \sigma\circ\gamma(m_i)=\hat\g(\varepsilon^\pm_{\sigma,\hat\g})^*\gamma(m_i)\,,
$$
where the $m_i$ are elements of the Pimsner-Popa basis from Proposition \ref{prop:discrete}.

We first note that $\gamma(m_i)^*\in{\rm Hom}(\hat\g^2,\hat\g)$. Hence, by naturality and tensoriality of the braidings $\varepsilon^{\pm}$ ({\it cf.} \eqref{naturality} and \eqref{tensoriality2}),
$$
\varepsilon^{\pm}_{\sigma,\hat\g}\sigma(\gamma(m_i^*))=\gamma(m_i)^*\varepsilon^{\pm}_{\sigma,\hat\g^2}=\gamma(m_i)^*\hat\g(\varepsilon^{\pm}_{\sigma,\hat\g})\varepsilon^{\pm}_{\sigma,\hat\g}\,,
$$
and thus
$$
\ad\!{}_{\varepsilon^\pm_{\sigma,\hat\g}}\circ \sigma\circ \gamma(m_i)=\hat\g(\varepsilon^{\pm}_{\sigma,\hat\g})^*\gamma(m_i)\,.
$$

Since $\mathcal{M}$ is generated as a von Neumann algebra by $\mathcal{N}$ and $\{m_i\}_{i\in I}$ ({\it cf.} \cite{FidIso}, Proposition 3.6), and since $\gamma$ is normal, we get the claim.
\end{proof}
When $\mathcal{M}:=\mathcal{N}\rtimes_\alpha \mathbb{Z}$, with $\alpha$ an outer automorphism in $\ob(\mathcal{C})$, the alpha-induction extension $\alpha^{\pm}_\sigma$ coincides with $\Phi_{\sigma,\epsilon^\mp_{\alpha,\sigma}}$, as the following proposition shows explicitly.

\begin{prop}
Let $\mathcal{N}$ be a type $\ty{III}$ factor with separable predual, $\alpha\in\ed(\mathcal{\mathcal{N}})$ be an outer automorphism of $\mathcal{N}$, and $\mathcal{M}=\mathcal{N}\rtimes_\alpha \mathbb{Z}$. Let $\mathcal{C}$ be a full tensor subcategory  of $\ed(\mathcal{N})$ containing $\alpha$ and $\hat\g$, which is braided with unitary braidings $\varepsilon^{\pm}$. Then for any $\sigma\in\ob(\mathcal{C})$, 
$$\alpha^{\pm}_\sigma=\Phi_{\sigma,\epsilon^\mp_{\alpha,\sigma}}.$$
\end{prop}
\begin{proof}
It is enough to show that 
\begin{itemize}
\item[] $\alpha^\pm_\sigma(n)=\sigma(n)$ for all $n\in\mathcal{N}$,
\item[] $\alpha^\pm_\sigma(V)=\epsilon^{\mp}_{\alpha,\sigma}V$.
\end{itemize}
The first equality is straightforward, recalling that $\hat\g=\gamma\lceil_\mathcal{N}$.
For the second equality, note that, since $Vn=\alpha(n)V$ for all $n\in\mathcal{N}$, we have $\gamma(V)\in{\rm Hom}(\hat\g,\hat\g\circ\alpha)$. By naturality and tensoriality of the braiding, we get
$$
\gamma(V)\varepsilon^{\pm}_{\sigma,\hat\g}=\varepsilon^{\pm}_{\sigma,\hat\g\circ\alpha} \sigma(\gamma(V))=\hat\g(\varepsilon^{\pm}_{\sigma,\alpha})\varepsilon^{\pm}_{\sigma,\hat\g}\sigma(\gamma(V))\,,
$$
that is
\begin{equation}
\label{alphaproof}
\gamma(V)=\hat\g(\varepsilon^{\pm}_{\sigma,\alpha})\varepsilon^{\pm}_{\sigma,\hat\g}\sigma(\gamma(V))\varepsilon^{\pm *}_{\sigma,\hat\g}=\hat\g(\varepsilon^{\pm}_{\sigma,\alpha})\gamma(\alpha^{\pm}_\sigma(V)).
\end{equation}

Applying $\gamma^{-1}$ to \eqref{alphaproof} we get
$$\alpha^{\pm}_{\sigma}(V)=\varepsilon^{\pm *}_{\sigma,\alpha}V=\varepsilon^{\mp}_{\alpha,\sigma}V.$$
\end{proof}

\section{General ergodic properties of the skew-product}

We start with the following preliminary result concerning the topological ergodicity of $(\mathfrak{A} \rtimes_{\alpha}\mathbb{Z},\Phi_{\th,u})$, resulting from that of $(\ga,\th)$.

\begin{prop}
\label{toerp}
For the $C^*$-dynamical systems $(\ga,\th)$ and $(\ga\rtimes_\a\bz, \Phi_{\th,u})$, the following are equivalent:
\begin{itemize} 
\item[(i)] $(\ga\rtimes_\a\bz)^{\Phi_{\th,u}}=\ga^\th$;
\item[(ii)] with the $u_n$ given in \eqref{coo}, the equation 
\begin{equation}
\label{comc*}
\a^{-n}(u_n)\th(a)=a\,,\quad a\in\ga\,,
\end{equation}
has no nontrivial solutions for each $n\neq0$.
\end{itemize}
Consequently, suppose that $(\ga,\th)$ is topologically ergodic. Then $(\ga\rtimes_\a\bz, \Phi_{\th,u})$ is also topologically ergodic if and only if {\rm (ii)} holds true.
\end{prop}
\begin{proof}
As shown in \cite{D}, Theorem VIII.2.2, any element $x\in \ga\rtimes_\a\bz$ can be decomposed into its Fourier series
$\sum_{n\in\mathbb{Z}}V^nc_n(x)$, with $c_n(x):=E(V^{-n}x)$, and the convergence being understood in norm
in the sense of Ces\`aro. 
But then, for $x\in\ga\rtimes_\a\bz$,
$$
\Phi_{\th,u}(x)-x=\th(c_o(x))-c_o(x)+\sum_{n\neq0}V^n\big(\a^{-n}(u_n)\th(c_n(x))-c_n(x)\big)\,,
$$
and therefore $(\ga\rtimes_\a\bz)^{\Phi_{\th,u}}=\ga^\th$ turns out to be equivalent to the fact that $\a^{-n}(u_n)\th(a)=a$ has no nontrivial solutions for $n\neq0$.

The last assertion is an immediate consequence of the previous part.
\end{proof}
The next proposition provides more information on the structure of $W_n$, the set of all solutions of the cohomological equations \eqref{comc*} at level $n$ when $(\mathfrak{A}, \theta)$ is topologically ergodic.
\begin{prop}
\label{unco}
Suppose that the $C^*$-dynamical systems $(\ga,\th)$ is topologically ergodic. Then the sets $W_n$ of the solutions of the cohomological equations, $n\in\bz$, are generated by a single unitary: if \eqref{comc*} admits a nonnull solution, then
$W_n=\bc w_n$, for some  $w_n\in\ga$ with $w^*_nw_n=I=w_nw^*_n$.\footnote{Since $(\ga,\th)$ is topologically ergodic, $W_o=\bc I$.
In addition, if for some $n\neq0$, \eqref{comc*} has no nontrivial solution, $W_n$ is the zero-dimensional subspace generated by any unitary in $\ga$.}
\end{prop}
\begin{proof}
We first prove that all solutions are a, possibly trivial, multiple of a unitary. 

If $a\in\ga$ satisfies the equation
$\alpha^{-n}(u_n)\theta(a)=a$, then it also satisfies $\theta(a^*)\alpha^{-n}(u_n^*)=a^*$. Multiplying the latter by the former equation, we find the equality $\theta(a^*a)=a^*a$. Hence, by topological ergodicity we must have $a^*a=\lambda I$, for some, necessarily positive, scalar $\lambda$. Since $a$ is a solution of the equation, $V^na=\alpha^n(a)V^n$ is
$\Phi_{\theta, u}$- invariant, but then the equation $\theta(\alpha^n(a))u_n=\alpha^n(a)$ must be satisfied as well. 
Taking its adjoint, it yields $u_n^*\theta(\alpha^n(a^*))=\alpha^n(a^*)$. Again, if we multiply the former by the latter equality,
we now find $\theta(\alpha^n(aa^*))=\alpha^n(aa^*)$, whence $\alpha^n(aa^*)$ is a scalar. But then $aa^*$ is a scalar as well, say $aa^*=\mu I$. Since $\|a^*a\|=\|aa^*\|$, we have $\lambda=\mu$, which means $a$ is a multiple of a unitary.

All is left to do is show that any two (non-trivial) solutions associated with the same $n$, say $a,b\in\ga$, are linearly dependent. From
$\alpha^{-n}(u_n)\theta(a)=a$ and $\theta(b^*)\alpha^{-n}(u_n^*)=b^*$, we see that $\theta(b^*a)=b^*a$, hence
$b^*a=\lambda I$. Thanks to the previous part, we can safely assume that $b$ is a unitary, which means $b=\lambda a$.
\end{proof}
Let $(\ch_\om,\pi_\om, V_{\om,\Phi_{\th,u}},\xi_\om)$ be the GNS covariant representation of $\ga\rtimes_\a\bz$
associated to the invariant state $\om$. Consider the orthogonal projection $P$ onto the closed subspace
$[\pi_\om(\ga)\xi_\om]:=\overline{\pi_\omega(\mathfrak{A})\xi_\omega}$. It is easily seen that $P\in\big(\pi_\om(\ga)\bigcup\{V_{\om,\Phi_{\th,u}}\}\big)'$, and therefore $(P\ch_\om,P\pi_\om, PV_{\om,\Phi_{\th,u}},\xi_\om)$ provides, up to unitary equivalence, the GNS covariant representation $(\ch_{\om_o},\pi_{\om_o}, V_{\om_o,\th},\xi_{\om_o})$ of $\ga$ associated to the invariant state $\om$.
Under the above identification of the covariant GNS representation of $\om_o$,  we also note that $E_1^{V_{\om_o,\th}}\leq E_1^{V_{\om,\Phi_{\th,u}}}$.

If we set 
$$
\ch_o:=[\pi_\om(\ga)\xi_\om]=\ch_{\om_o}\,,\,\,\ch_n:=\pi_\om(V)^n[\pi_\om(\ga)\xi_\om]\,,\,\,\,n\in\bz\,.
$$
we have the following elementary, yet 
useful, result.
\begin{lem}
\label{orto}
The Hilbert spaces $\mathcal{H}_\omega$ decomposes as $\oplus_{n\in\mathbb{Z}} \mathcal{H}_n$.
\end{lem}
\begin{proof}
By cyclicity, all we need to prove is $\mathcal{H}_m\perp \mathcal{H}_n$ if $m$ and $n$ differ. To this end, it is enough
to show that any two vectors $\pi_\omega(V^ma)\xi_\omega$ and $\pi_\omega(V^nb)\xi_\omega$, $a,b\in\mathfrak{A}$, are orthogonal in $\mathcal{H}_\omega$. But this is seen by an easy computation, for
\begin{align*}
&\langle\pi_\omega(V^ma)\xi_\omega,\pi_\omega(V^nb)\xi_\omega\rangle=\om(b^*V^{m-n}a)\\
=&\om(b^*\th^{m-n}(a)V^{m-n})
=\om_o(b^*\th^{m-n}(a))\d_{m,n}\,.
\end{align*} 
\end{proof}
\begin{rem}
\label{czrffi}
The subspaces $\mathcal{H}_n$ are all invariant under the restriction of $\pi_\omega$ to $\mathfrak{A}$:
$$
\pi_\omega(\mathfrak{A})\mathcal{H}_n\subset \mathcal{H}_n\,,\quad n\in\bz\,.
$$
Furthermore, as noticed above, the restriction
of $\pi_\omega$, thought of as a representation of $\mathfrak{A}$, to the cyclic subspace $\mathcal{H}_o$ is
unitarily equivalent to the GNS representation $\pi_{\omega_o}$ of $\om_o$. 
\end{rem}
The following result is the version of Proposition \ref{toerp} at the Hilbert space level.
\begin{prop}
\label{toerph}
For the $C^*$-dynamical systems $(\ga,\th,\om_o)$ and $(\ga\rtimes_\a\bz, \Phi_{\th,u},\om)$, the following are equivalent:
\begin{itemize} 
\item[(i)] $E_1^{V_{\om_o,\th}}=E_1^{V_{\om,\Phi_{\th,u}}}$;
\item[(ii)] with the $u_n$ given in \eqref{coo}, the equation 
\begin{equation}
\label{coml2}
\pi_{\om_o}(\a^{-n}(u_n))V_{\om_o,\th}\eta=\eta\,,\quad\eta\in\ch_{\om_o}\,,
\end{equation}
has no nontrivial solutions for each $n\neq0$.
\end{itemize}

Consequently, suppose that $(\ga,\th,\om_o)$ is weakly clustering. Then $(\ga\rtimes_\a\bz, \Phi_{\th,u},\om)$ is also weakly clustering if and only if {\rm (ii)} holds true.
\end{prop}
\begin{proof}
By using the decomposition of $\ch_\om$ into orthogonal components in Lemma \ref{orto}, we have 
$$
\ch_\om\ni\xi=\xi_o+\sum_{n\neq0}\pi_\om(V^n)\xi_n\,,
$$ 
for a sequence 
$(\xi_n)_{n\in\bz}\subset\ch_{\om_o}\subset\ch_\om$. Consequently,
$$
V_{\om,\Phi_{\th,u}}\xi-\xi=V_{\om_o,\th}\xi_o-\xi_o+\sum_{n\neq0}\pi_\om(V^n)\big(\pi_{\om_o}(\a^{-n}(u_n))V_{\om_o,\th}-\idd_{\ch_{\om_o}}\big)\xi_n\,,
$$
and therefore $E_1^{V_{\om_o,\th}}=E_1^{V_{\om,\Phi_{\th,u}}}$ turns out to be equivalent to the fact that $\pi_{\om_o}(\a^{-n}(u_n))V_{\om_o,\th}\eta=\eta$ has no nontrivial solutions in $\ch_{\om_o}$ for $n\neq0$.

The weak clustering property of the invariant state $\om_o$ means that $\dim\big(E_1^{V_{\om_o,\th}}\big)=1$, and thus $\dim\big(E_1^{V_{\om,\Phi_{\th,u}}})=1$ if and only if {\rm (ii)} holds true.
\end{proof}

\section{Unique ergodicity and the cohomological equation}
 
In the present section, we suppose that the $C^*$-dynamical system $(\ga,\th,\om_o)$ is uniquely ergodic if it is not differently stated.
In analogy with what was found  in dealing with classical Anzai flows and skew-product on the noncommutative torus ({\it cf.} \cite{Fu, DFGR}), unique ergodicity 
is actually automatic for the $C^*$-dynamical system $(\ga\rtimes_\a\bz, \Phi_{\th,u},\om)$ as soon as it is ergodic.
\begin{prop}
\label{evsue}
The $C^*$-dynamical system $(\ga\rtimes_\a\bz, \Phi_{\th,u})$ is uniquely ergodic with $\omega\in\mathcal{S}(\ga\rtimes_\a\bz)$ 
its unique invariant state, if and only if  $(\ga\rtimes_\a\bz, \Phi_{\th,u},\om)$ is ergodic. 
\end{prop}
\begin{proof}
We need only prove that ergodicity implies unique ergodicity. 
We start by noticing that for the gauge automorphisms,
$$
\r_z\circ\Phi_{\th,u}=\Phi_{\th,u}\circ\r_z\,,\quad z\in\bt\,,
$$
and thus $\r_z\big(\mathcal{S}(\ga\rtimes_\a\bz)^{\Phi_{\th,u}}\big)=\mathcal{S}(\ga\rtimes_\a\bz)^{\Phi_{\th,u}}$, $z\in\bt$.

Now, we assume that $\om$ is an extreme invariant state.
The key property to point out is that the equality 
$$
\oint\f\circ\r_z\frac{\di z}{2\pi\imath z}=\om
$$
holds true for any $\f\in\mathcal{S}(\ga\rtimes_\a\bz)^{\Phi_{\th,u}}$. By a standard approximation argument,
it suffices to check the claimed
equality only on elements of the form $V^na$, $a\in\mathfrak{A}$ and $n\in\mathbb{Z}$.
To this end, first we note that $\f\lceil_\ga=\om_o$ by unique ergodicity. Then we have
\begin{align*}
\oint\f\big(\r_z(V^na)\big)\frac{\di z}{2\pi\imath z}=\f(V^na)\oint z^n\frac{\di z}{2\pi\imath z}=\f(V^na)\delta_{n,0}=\omega_0(a)\delta_{n,0}
\end{align*}

Now the conclusion follows by the same argument employed in \cite{DFGR}, Theorem 4.5. 
\end{proof}
We can now move onto the main theorem of the present section.
\begin{thm}
\label{mier}
Let $(\ga,\th)$ be uniquely ergodic with $\om_o\in\cs(\ga)$ as the unique invariant state. Then the following assertions are equivalent:
\begin{itemize}
\item[(i)] the cohomological equation \eqref{coml2} has no nontrivial solution for $n\neq0$;
\item[(ii)] the $C^*$-dynamical system $(\ga\rtimes_\a\bz, \Phi_{\th,u},\om)$ is weakly clustering;
\item[(iii)] the $C^*$-dynamical system $(\ga\rtimes_\a\bz, \Phi_{\th,u},\om)$ is ergodic;
\item[(iv)] the $C^*$-dynamical system $(\ga\rtimes_\a\bz, \Phi_{\th,u})$ is uniquely ergodic with $\omega\in\mathcal{S}(\ga\rtimes_\a\bz)$ 
as its unique invariant state.
\end{itemize}
Suppose in addition that $s(\om_o)\in Z(\ga^{**})$, then one of the previous conditions {\rm(i)-(iv)} turns out to be equivalent to
\begin{itemize}
\item[(v)] with the $u_n$ given in \eqref{coo}, the equation 
\begin{equation}
\label{compi}
\pi_{\om_o}(\a^{-n}(u_n))\ad{}\!_{V_{\om_o,\th}}(v)=v\,,\quad v\in\pi_{\om_o}(\ga)''\,,
\end{equation}
has no nontrivial solutions for each $n\neq0$.
\end{itemize}
\end{thm}
\begin{proof}
(i)$\Longrightarrow$(ii) Since $(\ga,\th,\om_o)$ is uniquely ergodic, as we have seen in \eqref{xco}, we have
$$
\lim_{n\to+\infty}\Big(\frac1{n}\sum_{k=1}^{n-1}\om_o(a\th^k(b))\Big)=\om_o(a)\om_o(b)\,,\quad a,b\in\ga\,,
$$
and therefore it is weakly clustering. But this means $\dim\big(E_1^{V_{\om_o,\th}}\big)=1$. If (i) holds true, then by Proposition \ref{toerph},
$\dim\big(E_1^{V_{\om,\Phi_{\th,u}}}\big)=1$ as well, and therefore $(\ga\rtimes_\a\bz, \Phi_{\th,u},\om)$ is weakly clustering.

(ii)$\Longrightarrow$(iii) is clear since weak clustering implies ergodicity.

(iii)$\iff$(iv) is nothing else than the content of Proposition \ref{evsue}.

(iv)$\Longrightarrow$(i) Suppose that $(\ga\rtimes_\a\bz, \Phi_{\th,u},\om)$ is uniquely ergodic. Then
$$
\lim_{n\to+\infty}\Big(\frac1{n}\sum_{k=1}^{n-1}\om(a\Phi_{\th,u}^k(b))\Big)=\om(a)\om(b)\,,\quad a,b\in\ga\rtimes_\a\bz\,,
$$
and therefore it is weakly clustering. But this means $\dim\big(E_1^{V_{\om_o,\th}}\big)=1=\dim\big(E_1^{V_{\om,\Phi_{\th,u}}}\big)$. Thus (i) holds true again by Proposition \ref{toerph}.

(i)$\Longrightarrow$(v) is trivial, without any further assumption on the support of $\om_o$. Suppose now $s(\om_o)\in Z(\ga^{**})$. In this situation, if $\eta\in\ch_{\om_o}\smallsetminus\{0\}$ satisfies \eqref{coml2}, then by Lemmata  \ref{cif0} and \ref{cif1},
$$
|\eta|=\big|\pi_{\om_o}(\a^{-n}(u_n))V_{\om_o,\th}\eta\big|=\big|V_{\om_o,\th}\eta\big|=V_{\om_o,\th}|\eta|\,,
$$
and therefore $|\eta|$ is invariant. By Lemma \ref{cif2} and 
Proposition \ref{szcon}, we first deduce $|\eta|=c\xi_{\om_o}$, and then $\eta=cv_\eta\xi_{\om_o}$ for some partial isometry $v_\eta\in\pi_{\om_o}(\ga)''$, and some number $c>0$.

Therefore, if $\eta\neq0$, \eqref{coml2} leads to
$$
\big(\pi_{\om_o}(\a^{-n}(u_n))\ad{}\!_{V_{\om_o,\th}}(v_\eta)-v_\eta\big)\xi_{\om_o}=0\,,
$$
and thus (v)$\Longrightarrow$(i) follows since $\xi_{\om_o}$ is separating for $\pi_{\om_o}(\ga)''$.
\end{proof}
\begin{cor}
Suppose that one of the equivalent conditions {\rm(i)-(iv)} of Theorem \ref{mier} is satisfied. Then the skew-product $\F_{\th,u}\in\aut(\ga\rtimes_\a\bz)$ is minimal if 
$\om_o\in\cs(\ga)^\th$ is faithful.
\end{cor}
\begin{proof}
The assertion directly follows by \cite{LP}, Corollary 2.6, taking into account that the conditions {\rm(i)-(iv)} of Theorem \ref{mier} are equivalent to the unique ergodicity of 
$\F_{\th,u}$, and that $\om=\om_o\circ E$, $E$ given in \eqref{gcxe}, is faithful if and only if $\om_o$ is so.
\end{proof}
\begin{rem}
\label{rem:parisom}
For a uniquely ergodic $C^*$-dynamical system $(\ga,\th,\om_o)$ such that $s(\om_o)\in Z(\ga^{**})$, we have incidentally seen that each solution of the equation
\eqref{compi} is necessarily a multiple of a partial isometry. 
\end{rem}

Much more than pointed out in Remark \ref{rem:parisom} is actually true. Indeed, for a non necessarily uniquely ergodic $C^*$-dynamical system, we have
\begin{prop}
\label{unme}
Let  $(\ga, \theta, \omega_0)$ be a weakly clustering $C^*$-dynamical system.
If $s(\omega_0)\in Z(\ga^{**})$, then the solutions of \eqref{compi} are, possibly trivial, multiples
of a unitary.
\end{prop}
\begin {proof}
We start by claiming that, under our hypotheses, any $a\in\pi_{\om_o}(\ga)''$ such that $\ad{}\!_{V_{\om_o,\th}}(a)=a$ must be
a scalar multiple of the identity, \emph {i.e.} $a=\lambda I$, for some $\lambda\in\mathbb{C}$.
Now let $v\in\pi_{\om_o}(\ga)''$ be a solution of $\pi_{\om_o}(\a^{-n}(u_n))\ad{}\!_{V_{\om_o,\th}}(v)=v$, for some
$n\in\mathbb{Z}$. Taking the adjoint on both sides of the equality yields
 $\ad{}\!_{V_{\om_o,\th}}(v^*)\pi_{\om_o}(\a^{-n}(u_n^*))=v^*$. If we now multiply the latter by the former
equality, we get $\ad{}\!_{V_{\om_o,\th}}(v^*v)=v^*v$, which means $v^*v=\mu I$ for some positive
constant. Proceeding as in the proof of Proposition \ref{unco}, we see that $vv^*$ must be a multiple of the identity as well, which gives the conclusion. 
All is left to do is prove the claim made at the beginning.

Indeed, if $a\in\pi_{\om_o}(\ga)''$ is such that $\ad{}\!_{V_{\om_o,\th}}(a)=a$, then 
$V_ {\omega_0, \theta}\,a=aV_ {\omega_0, \theta}$. If we now compute this equality on
$\xi_ {\omega_0}$, we find $V_ {\omega_0, \theta}\,a\xi_{\omega_0}=a\xi_{\omega_0}$, and so
$a\xi_{\omega_0}=\lambda \xi_{\omega_0}$, for some $\lambda\in\mathbb{C}$, since
$\rm{Ran}\big(E_1^{V_{\omega_0,\theta}}\big)=\rm{Ker}(I-V_{\omega_0, \theta})=\mathbb{C}\xi_{\omega_0}$ thanks to the assumption that the system is weakly clustering. But, because the support of $\omega_0$ is central, the vector $\xi_{\omega_0}$ is also separating
for the von Neumann algebra $\pi_{\omega_0}(\ga)^{''}$, which means the above equality is possible if and only
if $a=\lambda I$, as maintained.
\end{proof}
A straightforward application of the faithfulness of $E$ given in \eqref{gcxe} and Proposition \ref{czni} yield
\begin{cor}
\label{inczfx}
Let $(\mathfrak{A}, \theta, \omega_0)$ be a $C^*$-dynamical system.
\begin{itemize}
\item[(i)] If $(\mathfrak{A}, \theta, \omega_0)$ is strictly ergodic, then 
$(\mathfrak{A}\rtimes_\alpha\mathbb{Z}, \Phi_{\theta, u}, \omega)$ is strictly ergodic if and only if \eqref{coml2} has no nontrivial solutions for $n\neq0$.
\item[(ii)] If $(\mathfrak{A}, \theta, \omega_0)$ is sharply ergodic, then 
$(\mathfrak{A}\rtimes_\alpha\mathbb{Z}, \Phi_{\theta, u}, \omega)$ is sharply
ergodic if and only if \eqref{compi} has no nontrivial solutions for $n\neq0$,
provided that $\om_o\in\cs(\ga)^\a$.
\end{itemize}
\end{cor}
Let $(\ga,\th,\om_o)$ be a $C^*$-dynamical system such that $(\ga,\th)$ is topologically ergodic. We recall that, by Proposition \ref{unco}, any nontrivial solution of the cohomological equation \eqref{comc*}, $n\in\bz$, is a multiple of a single unitary. Therefore, there is a one-to-one correspondence between solutions $v\in\ga$ satisfying \eqref{comc*} and their images $\pi_{\om_o}(v)\in\pi_{\om_o}(\ga)$.

In correspondence with any $C^*$-dynamical system $(\ga,\th,\om_o)$, we can recover another $C^*$-dynamical system $([\ga],[\th],[\om_o])$, where 
$[\ga]:=\ga/\ker(\pi_{\om_o})$, 
$$
[\th]([a]):=\th(a)\,,\,\,[\om_o]([a]):=\om_o(a)\,,\quad a\in[a]
$$
are well defined for each $[a]\in[\ga]$ because $\th(\ker(\pi_{\om_o}))=\ker(\pi_{\om_o})$, and 
$$
\ker(\pi_{\om_o})\subset\gpn_{\om_o}:=\{a\in\ga\mid\om_o(a^*a)=0\}.
$$
In such a way, with 
$$
\pi_{[\om_o]}([a]):=\pi_{\om_o}(a) \,,\quad a\in[a]\,,
$$
for $[a]\in[\ga]$,
its GNS covariant representation $\big(\ch_{[\om_o]},\pi_{[\om_o]}, V_{[\om_o],[\th]},\xi_{[\om_o]}\big)$ coincides, up to unitary equivalence, with
$(\ch_{\om_o}, [\pi_{\om_o}],V_{\om_o,\th},\xi_{\om_o})$. If in addition $\om_o$ is $\a$-invariant, we also have $\a(\ker(\pi_{\om_o}))=\ker(\pi_{\om_o})$ and thus the cohomological equations are meaningful in $[\ga]$, obtaining for $n\in\bz$,
\begin{equation}
\label{comcq*}
[\a]^{-n}[(u_n)][\th]([v])=[v]\,. 
\end{equation}

Therefore, there is another one-to-one correspondence between solutions of \eqref{comc*} in $\ga$ and solutions of \eqref{comcq*} in $[\ga]$.

Putting the two things together, we have
\begin{rem}
For the topologically ergodic $C^*$-dynamical system $(\ga,\th,\om_o)$, there is a one-to-one correspondence
$$
\ga\ni v\leftrightarrow\pi_{\om_o}(v)\in\pi_{\om_o}(\ga)
$$
between solutions $v$ in $\ga$ of the \eqref{comc*} and their images $\pi_{\om_o}(v)$ in $\pi_{\om_o}(\ga)$.

If in addition $\om_o\in\cs(\ga)^\a$, there is another one-to-one correspondence
$$
\ga\ni v\leftrightarrow[v]\in[\ga]
$$
between solutions $v$ in $\ga$ of the \eqref{comc*}  and solutions $[v]$ in $[\ga]$ of  \eqref{comcq*}.
\end{rem}

\section{Unique ergodicity with respect to the fixed-point subalgebra}
\label{uniqueergo}

In analogy with the terminology used in \cite{DFGR, F4}, we start with the following definition relative to the splitting of the solutions of the cohomological equations \eqref{coml2} in 
\lq\lq continuous" and \lq\lq measurable non-continuous" ones.
\begin{defin}
\label{dcmn} 
Let $(\ga,\th,\om_o)$ be $C^*$-dynamical system, $\a\in\aut(\ga)$ satisfying \eqref{inter} for some $u\in\mathcal{U}(\ga)$, and for each fixed $n\in\bz$ consider the cohomological equation \eqref{coml2}. 

The solutions $\eta\in\ch_{\om_o}$
of \eqref{coml2} of the form $\eta=\pi_{\om_o}(a)\xi_{\om_o}$ for some $a\in\ga$, satisfying \eqref{comc*}, are called 
\lq\lq continuous", whereas the remaining ones are called \lq\lq measurable non-continuous".
\end{defin} 

For $n_o\in\bz$, we define $S_{n_o}\subset\ga\rtimes_\a\bz$ as
$$
S_{n_o}:=V^{n_o}\big\{a\mid a\in\ga\,\,\text{satisfying \eqref{comc*} for $n=n_o$}\big\}=V^{n_o}W_{n_o}\,,
$$
where the $W_{n_o}\subset\ga$ are made of all solutions of the coomological equation \eqref{comc*} for the given $n_o\in\bz$. We note that if $(\ga,\th)$ is topologically ergodic, then $W_{n_o}=\bc V^{n_o}w_{n_o}$
is generated by a single unitary $w_{n_o}$ ({\it cf.} Proposition \ref{unco}), as long as it is not trivial.
\begin{prop}
\label{str}
For the 
$C^*$-dynamical system $(\ga\rtimes_\a\bz,\F_{\th,u})$, we have for $m,n\in\bz$,
\begin{itemize}
\item[(i)] $S_m^*=S_{-m}$, $S_mS_n\subset S_{m+n}$;
\item[(ii)] $\big(\ga\rtimes_\a\bz\big)^{\F_{\th,u}}=\overline{{\rm span}\{S_n\mid n\in\bz\}}^{\|\,\,\,\|}$.
\end{itemize}
In addition, if $(\ga,\th)$ is topologically ergodic, which happens when $(\ga,\th)$ uniquely ergodic, then
\begin{equation*}
G_{\th,u}:=\big\{n\in\bz\mid S_n\neq\{0\}\big\}
\end{equation*}
is a subgroup of $\bz$.
\end{prop}
\begin{proof}
(ii) directly follows by the lines in Proposition \ref{toerp}. 
We also note that if $V^ma\in S_m$ and $V^nb\in S_n$, then 
\begin{align*}
(V^ma)^*&=V^{-m}(V^ma^*V^{-m})=V^{-m}\a^m(a^*)\,,\\
V^maV^nb&=V^{m+n}(V^{-n}aV^n)b=V^{m+n}\a^{-n}(a)b\,,
\end{align*}
and $\a^m(a^*)$, $\a^{-n}(a)b$ satisfy \eqref{comc*} for $-m$ and $m+n$, respectively. Therefore, (i) holds true.

Suppose now that $(\ga,\th)$ is topologically ergodic. By Proposition \ref{unco}, if $S_n\neq\{0\}$, it is generated by the single unitary $V^nw_n$ and therefore $G_{\th,u}$, seen as a subset of $\bz$, is closed by taking sums and opposite by (i), that is it is a subgroup of $\bz$.
\end{proof}
The previous considerations tell us that if $(\ga,\th)$ is topologically ergodic, then $G_{\th,u}=l_{\th,u}\bz$ for some $l_{\th,u}\in\bn$. Note that $l_{\theta, u}$ is nothing but  
$\min\{n>0\mid S_n\neq \{0\}\}$.

In the case when $(\ga,\th)$ is topologically ergodic, we define a linear map
$E_{\F_{\th,u}}$ from the crossed product $C^*$-algebra to itself by setting
\begin{equation}
\label{geu}
E_{\F_{\th,u}}(x):=\sum_{l\in G_{\th,u}}\om_o\big(w_{-l}E(V^{-l}x)\big)V^lw_l\,,\quad x\in\mathfrak{A}\rtimes_{\alpha}\mathbb{Z}\,,
\end{equation}
where $E$ is the conditional expectation in \eqref{gcxe}, the $w_l$ are the unique unitary generating the $W_l$, and the convergence of the infinite sum is meant in norm in the sense of Ces\`aro.

If we define
$$
\widetilde{w}_n:=\left\{\begin{array}{ll}
                     w_n & \!\!\text{if}\,\, n\in G_{\th,u}\,, \\
                      0 & \!\!\text{otherwise}\,,
                    \end{array}
                    \right.
$$
\eqref{geu} becomes
$$
E_{\F_{\th,u}}(x):=\sum_{n\in\bz}\om_o\big(\widetilde w_{-n}E(V^{-n}x)\big)V^n\widetilde w_n\quad x \in\mathfrak{A}\rtimes{_\alpha} \mathbb{Z}\,.
$$
We also define
$$
\ga_{\th,u}:=\overline{{\rm span}\big\{V^m \mathfrak{A}\mid m\in G_{\th,u}\big\}}^{\|\,\,\|}\,.
$$

It is obvious that $\ga_{\th, u}$ is a a $C^*$-subalgebra of $\ga\rtimes_\a\bz$ containing the fixed-point subalgebra $(\ga\rtimes_\a\bz)^{\Phi_{\theta, u}}$ of the automorphism $\Phi_{\theta, u}$.
\begin{prop}
\label{efwpa}
The map in \eqref{geu} is a conditional expectation onto the fixed-point subalgebra $(\mathfrak{A}\rtimes_{\alpha}\mathbb{Z})^{\F_{\th,u}}$ satisfying $E_{\F_{\th,u}}\circ \F_{\th,u}=E_{\F_{\th,u}}$.

If in addition $(\ga,\th)$ is uniquely ergodic, for each $x\in\ga_{\th,u}$ we have
$$
\lim_{n\to+\infty}\Big(\frac1{n}\sum_{k=0}^{n-1}\F^k_{\th,u}(x)\Big)=E_{\F_{\th,u}}(x)\,,
$$
in the norm topology.
\end{prop}
\begin{proof}
We leave to the reader the task to check that $E_{\F_{\th,u}}$ is a well-defined conditional expectation onto $(\ga\rtimes_\a\bz)^{\F_{\th,u}}$.

Concerning the invariance, by a standard approximation argument, it is enough to consider $x=V^ma$ for a fixed $m\in\bz$ and $a\in\ga$. For such an $x$ we have
$$
E_{\F_{\th,u}}(x)=\sum_{l\in G_{\th,u}}\d_{l,m}\om_o(w_{-l}a)V^lw_l\,.
$$

Taking into account that, for general skew-products, $V^{-m}\F^k_{\th,u}(V^m)\in\ga$ and, by \eqref{coo1}, for $l\in G_{\th,u}$, $V^{-l}\F^k_{\th,u}(V^l)=w_l\th^k(w_{-l})$,
we get
\begin{align*}
E_{\F_{\th,u}}(\F_{\th,u}(x))=&\sum_{l\in G_{\th,u}}\om_o\big(w_{-l}E(V^{-l}\F_{\th,u}(V^m)\th(a))\big)V^lw_l\\
=&\sum_{l\in G_{\th,u}}\om_o\big(w_{-l}E(V^{m-l}V^{-m}\F_{\th,u}(V^m)\th(a))\big)V^lw_l\\
=&\sum_{l\in G_{\th,u}}\d_{l,m}\om_o\big(w_{-l}V^{-l}\F_{\th,u}(V^l)\th(a))\big)V^lw_l\\
=&\sum_{l\in G_{\th,u}}\d_{l,m}\om_o\big(w_{-l}w_l\th(w_{-l})\th(a))\big)V^lw_l\\
=&\sum_{l\in G_{\th,u}}\d_{l,m}\om_o(w_{-l}a)V^lw_l\\
=&E_{\F_{\th,u}}(x)\,.
\end{align*}

Concerning the invariance, it is again enough to verify the claimed equality on the generators $x=V^{m}a$, for some $m\in G_{\th,u}$, and compute
\begin{align*}
&\frac1{n}\sum_{k=0}^{n-1}\F_{\th,u}^k(x)=\frac1{n}\sum_{k=0}^{n-1}\F^k_{\th,u}(V^{m})\th^k(a)\\
=&V^{m}w_{m}\Big(\frac1{n}\sum_{k=0}^{n-1}\th^k(w_{-m}a)\Big)
\rightarrow\om_o(w_{-m}a)V^{m}w_{m}\\
=&\sum_{l\in G_{\th,u}}\d_{l,m}\om_o(w_{-l}a)V^lw_l
=E_{\F_{\th,u}}(x)\,.
\end{align*}
\end{proof}
\begin{rem}
\label{mov}
From the above result, we immediately see that,
if $(\ga,\th)$ is uniquely ergodic and $G_{\th,u}=\bz$, which happens when the cohomological equation has a continuous non-trivial solution for $n=1$, then the $C^*$-dynamical system $(\ga\rtimes_\a\bz,\F_{\th,u})$ is uniquely ergodic w.r.t. the fixed-point subalgebra because in this case $\mathfrak{A}_{\theta, u}$ coincides the whole crossed product 
$\ga\rtimes_\a\bz$.
\end{rem}

We can now move on to prove the necessary condition for our dynamical system $(\ga\rtimes_\a\bz,\F_{\th,u}\big)$
to be uniquely ergodic w.r.t. the fixed-point subalgebra. To do so, we need a preliminary result worthy of a statement
to itself, which is essentially a version of Proposition \ref{efwpa} at the Hilbert space level.
\begin{prop}
\label{ueif}
Let $(\ga,\th,\om_o)$ be a uniquely ergodic $C^*$-dynamical system. Then for $\xi\in[\pi_\om(\ga_{\th,u})\xi_{\om_o}]$,
\begin{equation}
\label{pzfp}
\lim_{n\to+\infty}\frac1{n}\sum_{k=0}^{n-1}V_{\om,\F_{\th,u}}^k\xi=\sum_{l\in\bz}\big\langle\xi,\pi_\om(V^l\widetilde{w}_l)\xi_{\om_o}\big\rangle\pi_\om(V^l\widetilde{w}_l)\xi_{\om_o}\,.
\end{equation}

If in addition $\big(\ga\rtimes_\a\bz,\F_{\th,u}\big)$ is uniquely ergodic w.r.t. the fixed-point subalgebra, then \eqref{pzfp} holds true for each 
$\xi\in\ch_\om$.
\end{prop}
\begin{proof}
We start by noticing that the $P_{\ch_m}$ reduces $V_{\om,\F_{\th,u}}$:
$$
\big[V_{\om,\F_{\th,u}},P_{\ch_m}\big]=0\,,\quad m\in\bz\,.
$$
Therefore, by a standard approximation argument, we need only 
consider vectors of the form $\xi=\pi_{\om}(V^m)\eta\in\ch_m=\pi_{\om}(V^m)\ch_o$  for a fixed $m\in\bz$, where $\eta$ sits in $\mathcal{H}_0$. 

In order to prove the first statement, the integer $m$ can be supposed to lie in 
$G_{\theta, u}$. To this aim, for any 
$\eps>0$, we can choose $a_\eps\in\ga$ such that if we set
$$\eta_\varepsilon:=\pi_\omega(a_\varepsilon)\xi_{\omega_0}\quad \xi_\eps:=\pi_\omega(V^m)\eta_\eps$$
we have
$\|\xi-\xi_\eps\|=\|\eta-\eta_\eps\|<\eps$.
Thanks to an $\frac{\varepsilon}{3}$-argument, we then see that
\begin{align*}
&\Big\|\frac1{n}\sum_{k=0}^{n-1}V_{\om,\F_{\th,u}}^k\xi
-\sum_{l\in\bz}\big\langle\xi,\pi_\om(V^l\widetilde{w}_l)\xi_{\om_o}\big\rangle\pi_\om(V^l\widetilde{w}_l)\xi_{\om_o}\Big\|\\
=&\Big\|\frac1{n}\sum_{k=0}^{n-1}\pi_{\om}(V^m w_m\th^k(w_{-m}))V_{\om_o,\th}^k\eta
-\langle\eta,\pi_{\om_0}(w_m)\xi_{\om_o}\rangle\pi_{\om_o}(V^m w_m)\xi_{\om_o}\Big\|\\
<&2\eps+\Big\|\pi_{\om}(V^mw_m)\pi_{\om_o}\Big(\frac1{n}\sum_{k=0}^{n-1}\th^k(w_{-m}a_\eps)-\om_o(w_{-m}a_\eps) I\Big)\xi_{\om_o}\Big\|\,,
\end{align*}
which goes to $0$, as $\eps$ is arbitrary.

In order to prove the second statement, we also need to take into account integers
$m$ that do not sit in $G_{\theta, u}$, and for such $m$'s we have
$$
\lim_{n\to+\infty}\frac1{n}\sum_{k=0}^{n-1}V_{\om,\F_{\th,u}}^k\xi=0\,,\quad \xi\in[\pi_\om(\ga_{\th,u})\xi_{\om_o}]^\perp\,.
$$
Indeed, for $a\in\ga$, by unique ergodicity w.r.t. the fixed-point subalgebra we have
$$
\frac1{n}\sum_{k=0}^{n-1}\F^k_{\th,u}(V^{m}a)\rightarrow0
$$
in the $C^*$-norm of the crossed product algebra, as $n\to+\infty$. But then the same computation as before leads to
\begin{align*}
\Big\|\frac1{n}\sum_{k=0}^{n-1}V_{\om,\F_{\th,u}}^k\xi\Big\|
\leq&\Big\|\frac1{n}\sum_{k=0}^{n-1}V_{\om,\F_{\th,u}}^k(\xi-\xi_\eps)\Big\|
+\Big\|\frac1{n}\sum_{k=0}^{n-1}\F^k_{\th,u}(V^{m}a)\Big\|\\
<&\eps+\Big\|\frac1{n}\sum_{k=0}^{n-1}\F^k_{\th,u}(V^{m}a)\Big\|\rightarrow\eps\,,
\end{align*}
and the conclusion is thus arrived at.
\end{proof}
\begin{prop}
\label{necessarycond}
Let $(\ga\rtimes_\a\bz,\F_{\th,u})$ be a skew-product $C^*$-dynamical system based on the uniquely ergodic $C^*$-dynamical system $(\ga,\th,\om_o)$.  

If $(\ga\rtimes_\a\bz,\F_{\th,u})$ is uniquely ergodic w.r.t. the fixed-point subalgebra, then for each $n\in\bz$, \eqref{coml2} has no nontrivial solutions 
$\eta\in\ch_{\om_o}\smallsetminus\{\pi_{\om_o}(a)\xi_{\om_o}\mid a\in \ga\,\,\text{satisfying \eqref{comc*}}\}$.
\end{prop}
\begin{proof}
We suppose that $(\ga\rtimes_\a\bz,\F_{\th,u})$ is uniquely ergodic w.r.t. the fixed-point subalgebra. 

If for a fixed $m$, $\eta\in\ch_{\om_o}\smallsetminus\{\pi_{\om_o}(a)\xi_{\om_o}\mid a\in \ga\,\,\text{satisfying \eqref{comc*}}\}$ is a solution of the cohomological
equation, the vector $\xi=\pi_\om(V^m)\eta\in \mathcal{H}_\omega$  is invariant under $V_{\om,\F_{\th,u}}$. 

By applying Proposition \ref{pzfp}, we then find
\begin{align*}
\xi=V_{\om,\F_{\th,u}}\xi=&\frac1{n}\sum_{k=0}^{n-1}V_{\om,\F_{\th,u}}^k\xi
\rightarrow\sum_{n\in\bz}\big\langle\xi,\pi_\om(V^n\widetilde{w}_n)\xi_{\om_o}\big\rangle\pi_\om(V^n\widetilde{w}_n)\\
=&\sum_{n\in\bz}\d_{m,n}\langle\eta,\pi_{\om_o}(\widetilde{w}_n)\xi_{\om_o}\rangle\pi_\om(V^n\widetilde{w}_n)\xi_{\om_o}\\
=&\langle\eta,\pi_{\om_o}(\widetilde{w}_m)\xi_{\om_o}\rangle\big(\pi_\om(V^m)\pi_{\om_o}(\widetilde{w}_m)\xi_{\om_o}\big)\,.
\end{align*}

We deduce that $\xi$ is a multiple of $\pi_\om(V^m)\pi_{\om_o}(\widetilde{w}_m)\xi_{\om_o}$, which is actually $0$ if $m\notin G_{\th,u}$.

Therefore, for each fixed $n\in\bz$, \eqref{coml2} has no non-trivial solution 
$\eta\in\ch_{\om_o}\smallsetminus\{\pi_{\om_o}(a)\xi_{\om_o}\mid a\in \ga\,\,\text{satisfying \eqref{comc*}}\}$.
\end{proof}
We would expect the converse to Proposition \ref{necessarycond} to hold true as well. Despite our best efforts, though, we have not been able to prove it in full generality. However, we next show this is actually  the case for general classical Anzai skew-products of the type extensively analised by Furstenberg in \cite{Fu}, which we recalled in the introduction.

We indeed start with a compact Hausdorff space $X_o$, a homeomorphism
$\theta\in\textrm{Homeo}(X_o)$, and a $\theta$-invariant probability measure  $\mu_o$ on
$X_o$, that is the unique regular Borel probability measure $\mu_o$ such that
$\mu_o(\theta^{-1}(B))=\mu_o(B)$ for every Borel subset $B\subset X_o$, or equivalently by the Riesz-Markov Theorem,
$$
\int_{X_o}(h\circ\theta)\textrm{d}\mu=\int_{X_o} h\textrm{d}\mu_o\,,\quad h\in C(X_o)\,.
$$

As recalled in the introduction, corresponding to a given continuous function $f\in C(X_o;\mathbb{T})$, we  consider the Anzai skew-product (called also a process on the torus in \cite{Fu})
$\Phi_{\theta, f}\in\textrm{Homeo}(X_o\times\mathbb{T})$ as
$$
\Phi_{\theta, f}(x, w):=(\theta(x), f(x)w)\,,\quad (x, w)\in X_o\times\mathbb{T}\,.
$$ 

Many ergodic properties of the corresponding skew-product dynamical system
can  be described in terms of the cohomological equations, one for each $n\in\mathbb{Z}$, which here take the much simpler forms
\begin{equation}
\label{sifo}
g(\theta(x))f(x)^n=g(x)\,,\,\,\mu_o\,\text{-}\,\textrm{a.e.}\,,
\end{equation}
in the unknown complex function $g\in L^\infty(X_o, \mu_o)$, together with the twin ones
\begin{equation}
\label{sifo1}
g(\theta(x))f(x)^n=g(x)\,,
\end{equation}
where now the unknown is a function $g\in C(X_o)$.

In this simple situation, we consider the representation of $h\in C(X_o)$ by the multiplication operator $\pi_{\m_o}(h)$ acting on $L^2(X_o, \mu_o)$ as
$$
(\pi_{\m_o}(h)\xi)(x):=h(x)\xi(x)\,,\quad \xi\in L^2(X_o, \mu_o)\,,
$$
and $\pi_{\m_o}(h)\xi_{\m_o}(x)=h(x)1\in L^2(X_o,\m_o)$, as $\xi_{\m_o}(x)=1$ the constant function on $X_o$.
Therefore, solutions of \eqref{sifo} of the form $g=\pi_{\m_o}(G)$, with $G$ satisfying \eqref{sifo1}, are exactly the \lq\lq continuous solutions" according to Definition \ref{dcmn}. The reader is referred to Example \ref{ex0} for a simple, but illustrative situation.

The following result proves that also the converse of Proposition \ref{necessarycond} holds true for processes on the torus.\footnote{In the remaining part of the present section, as well as in some examples in Section \ref{ses}, for the skew-product we are using the notation borrowed from  \cite{DFGR, Fu}. Namely, $\th$ corresponds to the automorphism of $C(\bt)$ generated by the rotation of the angle $\th$, and $f$ corresponds to 
$u\in C(\bt)$ given by the continuous function $f(z)$ in $C(X_o;\bt)$, or $f(U)$ in the noncommutative 2-torus generated by the noncommutative coordinates $U,V$. With a minor abuse of notation, we use the symbol $\F_{\th,f}$ to denote both the action
 on $X_o\times\bt$ and the corresponding automorphism of $C(X_o\times T)$, \it{i.e.} $\Phi_{\theta, f}(h)= h\circ \Phi_{\theta, f}$, $h\in C(X_o\times\bt)$.}
\begin{thm}
\label{classicalue}
Let $(X_o,\th,\m_o)$ be a uniquely ergodic classical dynamical system based on the compact Hausdorff space $X_o$, $\th\in{\rm Homeo}(X_o)$ and the unique invariant regular Borel probability measure 
$\m_o$ under $\th$, together with $f\in C(X_o;\bt)$.  

The associated process on the torus 
$(X_o\times\bt,\Phi_{\theta, f})$ is uniquely ergodic w.r.t. the fixed-point subalgebra if and only if each solution $g$ of \eqref{sifo} is of the form $g=\pi_{\m_o}(G)$, with $G\in C(X_o)$ satisfying \eqref{sifo1}.\footnote{Roughly speaking, a process on the torus is uniquely ergodic w.r.t. the fixed-point algebra if and only if each``measurable''({\it i.e.} of the form \eqref{sifo}) solution of the cohomological equations is indeed ``continuous'' ({\it i.e.} of the form \eqref{sifo1}).}
\end{thm}
\begin{proof}
In Proposition \ref{necessarycond}, we have already proved for the general cases under consideration one side of the the assertion, and therefore we limit the analysis to the remaining direction.

As  was shown in Proposition \ref{str}, the fixed-point subalgebra is given by
\begin{equation}
\label{spf}
C(X_o\times\mathbb{T})^{\Phi_{\theta, f}}=\overline{\textrm{span}\{G_l(x,z)=g_{ln_o}(x)z^{l n_o}\mid l \in\mathbb{Z}\}}\,,
\end{equation}
where $n_o$ is the minimum of the strictly positive integers $n$ such that there exists a non-null
$g\in C(X_o)$ (indeed a multiple of a unitary function, uniquely determined up to a phase) with $g(\theta(x))f(x)^n=g(x)$ for all $x\in X_o$. 

We will prove unique ergodicity w.r.t. the fixed-point subalgebra by showing that condition  (ii) in Definition \ref{abdy}
is fulfilled.
In other terms, all we need to do is to check that any $\Phi_{\theta, f}$-invariant state on
$C(X_o\times\mathbb{T})$ is uniquely determined by its restriction to $C(X_o\times\mathbb{T})^{\Phi_{\theta, f}}$.

For such a purpose, fix the state $\varphi$ on $C(X_o\times\mathbb{T})$ invariant under $\Phi_{\theta, f}$, together with
 the corresponding regular Borel probability measure $\mu$ on the product space $X_o\times\mathbb{T}$, such that 
$\varphi(h)=\int_{X_o\times\mathbb{T}} h\textrm{d}\mu$ for every $h\in C(X_o\times\mathbb{T})$.

By virtue of the Stone-Weierstrass Theorem, the linear span of the set of functions 
of the form $k(x)z^l$, $k\in C(X_o)$ and $l\in\bz$, is a norm dense $*$-subalgebra of
$C(X_o\times \bt)$. In particular, the state $\varphi$ is uniquely determined by the values of the integrals
$\int_{X_o\times\mathbb{T}} k(x)z^l\textrm{d}\mu(x,z)$, where $k$ is any continuous function
on $X_o$, and $l$ any integer number. Therefore, it is enough to show to that,  for any
$k\in C(X_o)$ and $l\in\mathbb{Z}$, the value of the integral $\int_{X_o\times\mathbb{T}} k(x)z^l\textrm{d}\mu(x,z)$ actually depends only on the restriction of $\varphi$ to the fixed-point subalgebra. 

For such a purpose, there are two cases depending on whether $l$ is a multiple of $n_o$ or not.
If it is so, say $l=mn_o$, we can compute  $\int_{X_o\times\mathbb{T}} k(x)z^{mn_o}\textrm{d}\mu(x,z)$
in the following way. Let us denote by $V\in C(X_o\times\bt)$  the function $V(x, z):=z$, $(x, z)\in X_o\times\bt$.\footnote{The function $V(x, z)=z$ is nothing 
else than the outer unitary $V$ in Section \ref{cfprx} describing the (trivial) crossed product construction of $C(X_o\times \bt)$.}

Set $h(x, z):=k(x)z^{mn_o}$. Since $\Phi^j_{\theta, f}(g_{mn_o}V^{mn_o})=g_{mn_o}V^{mn_o}$ for any
$j\in\bn$, the Ces\`aro averages
of $h$ are given by
\begin{align*}
\bigg(\frac{1}{n}\sum_{j=0}^{n-1}h\circ&\Phi_{\theta, f}^j\bigg)(x, z)=\bigg(\frac{1}{n}\sum_{j=0}^{n-1}
\Phi^j_{\theta, f}(k V^{mn_o})\bigg)(x, z)\\
=&\bigg(\frac{1}{n}\sum_{j=0}^{n-1}\Phi^j_{\theta, f}(k \overline{g_{mn_o}}g_{mn_o}V^{mn_o})\bigg)(x, z)\\
=&\bigg(\frac{1}{n}\sum_{j=0}^{n-1} \Phi^j_{\theta, f}(k\overline{g_{mn_o}})\Phi^j_{\theta, f}(g_{mn_o}V^{mn_o})\bigg)(x, z)\\
=&\bigg(\frac{1}{n}\sum_{j=0}^{n-1}k(\theta^j(x))
\overline{g_{mn_o}(\theta^j(x))}\bigg)g_{mn_o}(x)z^{mn_o}\,.
\end{align*}
But $\theta$ is uniquely ergodic with $\mu_o$ being its unique invariant measure, which means that
the above Ces\`aro averages converge in norm to  $\langle k, g_{mn_o}\rangle_{L^2(X_o, \mu_0)}g_{mn_o}(x)z^{mn_o}$.
Since $\mu$ is $\Phi_{\theta, f}$-invariant, we have
\begin{align*}
\int_{X_o\times\bt} h{\rm d}\mu&=\int_{X_o\times\bt}\bigg(\frac{1}{n}\sum_{j=0}^{n-1}h\circ\Phi_{\theta, f}^j\bigg)\di\m\\
&=\lim_{n\rightarrow+\infty}\int_{X_o\times\bt}\bigg(\frac{1}{n}\sum_{j=0}^{n-1}h\circ\Phi_{\theta, f}^j\bigg)\di\m\\
&=\langle k, g_{mn_o}\rangle_{L^2(X_o, \mu_0)}\int_{X_o}g_{mn_o}(x)z^{mn_o}{\rm d}\mu_o.
\end{align*}

The conclusion is reached since the right-hand side of the equality above depends only on the
restriction of our state $\varphi$ to the fixed-point subalgebra as the function
$g_{mn_o}(x)z^{mn_o}$ is clearly
$\Phi_{\theta, f}$-invariant.

The second case is dealt with by showing that, if $l$ fails to be a multiple of
$n_o$, then $\int_{X_o\times\mathbb{T}}k(x)z^l\textrm{d}\mu(x,z)=0$ for any $k\in C(X_o)$.

We begin with the regular Borel measure $\m$ which is extreme under the action $\Phi_{\theta, f}$, that is $(X_o\times\mathbb{T}, \Phi_{\theta, f}, \mu)$ is  an ergodic dynamical system, and we shall argue by contradiction.

Let us suppose that there does exist a function $k\in C(X_o)$ such that 
$$
\int_{X_o\times\mathbb{T}}k(x)z^l\textrm{d}\mu(x,z)=\lambda\neq 0\,.
$$
We first recall that, for $(x,z)\in X_o\times\bt$, $V(x,z)=z$. Let us then define $h_{j,l}:=\Phi_{\theta, f}^j(V^l) V^{-l}\in C(X_o\times\bt)$, and note that $h_{j, l}$ actually depends only on the variable $x\in X_o$ as we have already shown. 

By Birkhoff's Ergodic Theorem ({\it e.g.} \cite{W}, Theorem 1.14), if we set $h(x, z):=k(x)z^l$, we see that the average 
$$
\frac{1}{n}\sum_{j=0}^{n-1}h(\Phi^j_{\theta, f}(x,z))\equiv\frac{1}{n}\sum_{j=0}^{n-1}k(\theta^j(x))h_{j,l}(x)z^l
$$ 
must converge to
$\lambda$, almost everywhere w.r.t. $\mu$. 

If 
$$
\ell_n(x):=\frac{1}{n}\sum_{j=0}^{n-1} k(\theta^j(x))h_{j,l}(x)\,,
$$
and $G_n(x,z)=\ell_n(x)z^l$, we have that the sequence $(G_n)_n$ converges to $\lambda$ almost everywhere w.r.t. $\mu$.

Let $C\subset X_o\times\mathbb{T}$ be the Borel set such that $\mu(C)=1$ and $(x, z)\in C$ implies that the numerical
sequence $\ell_n(x)z^l$ converges to $\lambda$. Note that the existence of such a set is provided again by
Birkhoff's Ergodic Theorem.

Let $A\subset X_o$ be the set of those $x$ such that
$\ell_n(x)$ converges to a non-zero complex number. We first note that $A$ is a Borel subset of $X_o$ since it can be rewritten as
the intersection $A= A_1\bigcap A_2\bigcap A_3$ with
\begin{align*}
&A_1:=\{x\in X_o\mid\limsup_ n {\rm Re}\, \ell_n(x)=\liminf_n {\rm Re}\,\ell_n(x)\}\,,\\
&A_2:=\{x\in X_o\mid\limsup_ n {\rm Im}\, \ell_n(x)=\liminf_n {\rm Im}\,\ell_n(x)\}\,,\\
&A_3:=\bigcup_{k=1}^\infty \{x\in X_o\mid\limsup_n |\ell_n(x)|>1/k\}\,,\\
\end{align*}
with $A_1$, $A_2$, and $A_3$ Borel subsets.
We then have $C\subset A\times\mathbb{T}$ because
$(x, z)\in C$ implies $\ell_n(x)$ converges to $\lambda z^{-l}$, whose absolute value is $|\lambda|$.
But then, $\mu(A\times\mathbb{T})\geq \mu (C)=1$, which means $\mu(A\times\mathbb{T})=1$.

Since  the marginal measure $\n(B):=\m(B\times\bt)$, $B\subset X_o$ measurable, is $\th$-invariant, and $(X_o, \theta, \mu_o)$ is uniquely ergodic,
we have $\mu(A\times\mathbb{T})=\mu_o(A)$. Therefore, the above equality now reads
$\mu_o(A)=1$, which means the sequence $(\ell_n)_n$ converges,
almost everywhere w.r.t. $\mu_o$, to a measurable function $\widetilde{g}$ on $X_o$, which by construction satisfies the cohomological
$\widetilde{g}(\theta(x))f(x)^l=\widetilde{g}(x)$, almost everywhere w.r.t $\mu_o$, and
$\widetilde{g}$ is certainly non-null. This leads to a contradiction and the assertion is proved for the ergodic regular Borel measures, or equivalently for the ergodic states on
$C(X_o\times\bt)$.

Pick now any regular Borel measure $\n$ on $C(X_o\times\bt)$, invariant under $\Phi_{\theta, f}$, and denote by $\f_\n$ the corresponding state. By the Krein-Milman Theorem, for 
fixed $F\in C(X_o\times\bt)$ with $F(x,z)=k(x)z^l$, $l$ failing to be a multiple of $n_o$, and $\eps>0$, there exists a finite set of ergodic regular Borel measures 
$\{\m_j\mid j=1,\dots,n\}$, and $\{\l_j\mid \l_j\geq0,\,\,\sum_{j=1}^n \l_j=1\}$, such that
$\big|\f_\n(F)-\sum_{j=1}^n \l_j\f_{\m_j}(F)\big|<\eps$.

We easily get by the previous part,
$$
|\f_\n(F)|\leq\big|\f_\n(F)-\sum_{j=1}^n \l_j\f_{\m_j}(F)\big|+\big|\sum_{j=1}^n \l_j\f_{\m_j}(F)\big|<\eps\,,
$$
and the assertion follows as $\eps$ is arbitrary.
\end{proof}
For the processes on the torus $(X_o\times\mathbb{T}, \Phi_{\theta, f})$ which are uniquely ergodic w.r.t. the fixed-point subalgebra, we can also provide a full description of all their invariant measures. The first thing to do is say what the fixed-point subalgebra
$C(X_o\times\mathbb{T})^{\Phi_{\theta, f}}$ given in \eqref{spf} looks like.
Up to isomorphism, there are actually only two possibilities, as shown below.
\begin{prop}
\label{fixpointchar}
Let $(X_o\times\mathbb{T}, \Phi_{\theta, f})$ be a process on the torus based on the topologically ergodic $C^*$-dynamical system $(X_o,\theta)$. Then
the fixed-point subalgebra $C(X_o\times\mathbb{T})^{\Phi_{\theta, f}}$
is $*$-isomorphic with either $\mathbb{C}$ or $C(\mathbb{T})$,
the first case occurring when $(X_o\times\mathbb{T}, \Phi_{\theta, f})$ is topologically ergodic.
\end{prop}
\begin{proof}
The case $C(X_o\times\mathbb{T})^{\Phi_{\theta, f}}=\bc I$ is trivial and occurs when the cohomological equations \eqref{comc*} have no nontrivial solutions for $n\neq0$.

The remaining case occurs when the cohomological equations \eqref{comc*} have nontrivial solutions for all and only $n$ in the subgroup of $\bz$ of the form $\{n_ol\mid l\in\bz\}$. In this case, the solutions are multiple of a single unitary 
$g_{n_ol}$ ({\it cf.} Proposition \ref{unco}) satisfying $\overline{g_{n_ol}}=g_{-n_ol}$, and in addition one can choose the $g_{n_ol}$ such that  
$g_{n_ol}g_{n_om}=g_{n_o(l+m)}$, see Proposition \ref{str}. 

We point out that the continuous functions $G_l(x,z):=g_{n_ol}(x)z^{n_ol}$ on $X\times\bt$ are linearly independent and
their linear span provides a (nontrivial) $*$-subalgebra of $C(X_o\times\mathbb{T})^{\Phi_{\theta, f}}$. 

We now observe that 
$$
\mathcal{A}:=\Big\{\sum_{m\in\mathbb{Z}} c_lG_l\mid\sum_{m\in\mathbb{Z}} |c_l|<\infty\Big \}\,,
$$
endowed with the $\ell_1$-type norm
$$
\Big[\!\!\Big]\sum_{l\in\mathbb{Z}}c_lG_l\Big[\!\!\Big]:=\sum_{z\in\mathbb{Z}}|c_l|\,,
$$
is also contained in  $C(X_o\times \mathbb{T})^{\Phi_{\theta, f}}$ since any series of that type
converges uniformly to a continuous $\Phi_{\theta, f}$-invariant function.

Thanks to the equality $G_lG_m=G_{l+m}$,
$\mathcal{A}$ is seen at once to be isometrically isomorphic with the Banach algebra $\ell_1(\mathbb{Z})$ understood
as the convolution algebra of $\mathbb{Z}$. 

Since the latter has only one $C^*$-completion, that is $C(\mathbb{T})$, we finally see
that the fixed-point subalgebra must be $*$-isomorphic with $C(\mathbb{T})$.
\end{proof}
Let $\mathcal{P}(\mathbb{T})\subset C(\bt)^*$ be the convex, $*$-weakly compact set of the regular Borel probability measures on the unit circle.\footnote{By the Riesz-Markov Theorem, we are identifying the set of finite signed Borel measures on $\bt$ with the dual of $C(\bt)$.}
Denote by 
$\r:C(X_o\times\mathbb{T})^{\Phi_{\theta, f}}\rightarrow C(\mathbb{T})$ the $C^*$-isomorphism provided by Proposition \ref{fixpointchar}.
\begin{prop}
Let $(X_o,\th)$ be a uniquely ergodic classical $C^*$-dynamical system, and suppose that the process on the torus $\Phi_{\theta, f}$ is uniquely ergodic w.r.t. the fixed-point algebra with $C(X_o\times\mathbb{T})^{\Phi_{\theta, f}}\supsetneq\bc I$. 

Then the convex, $*$-weakly compact set of the $\Phi_{\theta, f}$-invariant regular Borel probability measures is 
in bijection with $\mathcal{P}(\mathbb{T})$ through the affine map $T$ given by
$$
T(\mu):=\r^{\rm t}(\mu)\circ E^{\Phi_{\theta, f}}\,,\mu\in\mathcal{P}(\mathbb{T})\,,
$$
$E^{\Phi_{\theta, f}}$ being the unique $\Phi_{\theta, f}$-invariant conditional expectation of $C(X_o\times\mathbb{T})$ onto $C(X_o\times\mathbb{T})^{\Phi_{\theta, f}}$.\footnote{ Such a conditional expectation exists by \cite{AD}, Theorem 3.1.}
\end{prop}
\begin{proof}
To begin with, the map $T$ is of course injective. Furthermore, it is also surjective. Indeed, let 
$\sigma$ be any $\Phi_{\theta, f}$- invariant measure on $X_o\times\mathbb{T}$. 
Let us denote by $\varphi$ the corresponding state, that is $\varphi(h):=\int_{X_o\times \mathbb{T}} h\textrm{d}\sigma$,
$h\in C(X_o\times\mathbb{T})$.

We note that, with $\widetilde\f:=\f\lceil_{C(X_o\times\mathbb{T})^{\Phi_{\theta, f}}}$, $\widetilde\f\circ E^{\Phi_{\theta, f}}$
is a $\Phi_{\theta, f}$-invariant extension of $\tilde{\varphi}$ to the whole $C^*$-algebra $C(X_o\times\mathbb{T})$.

Since $(X_o\times \mathbb{T}, \Phi_{\theta, f})$ is uniquely ergodic w.r.t. the fixed-point subalgebra, we must have
$\varphi=\widetilde\f\circ E^{\Phi_{\theta, f}}$, and the proof follows.
\end {proof}
By using \eqref{geu}, we easily get for $F\in C(X_o\times\bt)$ and $\m\in\cp(\bt)$,
$$
T(\m)(F)=\sum_{l\in\bz}\widecheck{\m}(l)\int_{X_o}\di\m_o(x)\overline{g_{n_ol}(x)}\int_\bt\frac{\di z}{2\pi\imath z}F(x,z)z^{-n_ol}\,,
$$
where $\widecheck{\m}(l)=\int_\bt z^l\di\m(z)$ is the characteristic function of the measure $\m$, and the sum is meant in the sense of Ces\`aro.

\section{Examples}
\label{ses}

We intend to provide some illustrative, perhaps simple, examples of commutative and noncommutative skew-products.

\begin{examp}
\label{exax}
The double rotation on the two-dimensional torus $R_{\theta, \varphi}\in\textrm{Homeo}(\mathbb{T}\times\mathbb{T})$
given by 
\begin{equation}
\label{drt}
R_{\theta, \varphi}(z, w):= (e^{\imath\theta}z, e^{\imath\varphi}w)\,,\quad (z, w)\in\mathbb{T}\times\mathbb{T}\,,
\end{equation}
with
$\frac{\theta}{2\pi}$ irrational and $e^{\imath\varphi}$ a root of unity, provides the simplest example
of an Anzai skew-product that is uniquely ergodic w.r.t the fixed-point subalgebra while not being uniquely ergodic.

Indeed, if $\varphi=\frac{2\pi}{l}$,  for some positive natural number $l$, then the associated cohomological equation \eqref{sifo} at the level $n$ is
$g(e^{\imath\theta}z)e^{\frac{2\pi\imath n}{l}}=g(z)$, {\it a.e.} w.r.t. the Lebesgue-Haar measure of $\mathbb{T}$.

Let $U\in\cu(L^2(\bt,\l))$ be the Koopman operator associated to the rotation of the angle $\th$. It is well-known that $\s_{\rm pp}(U)=\{e^{\imath m\th}\mid m\in\bz\}$.\footnote{For a bounded operator $A$ acting on a Hilbert space, $\s_{\rm pp}(A)$ is denoting its pure-point spectrum.}
Since $\frac{\theta}{2\pi}$ is irrational and $e^{-\frac{2\pi\imath n}{l}}$ should belong to $\s_{\rm pp}(U)$ if the cohomological equation would admit a nontrivial solution, it can happen if and only if $e^{\frac{2\pi\imath n}{l}}=1$, providing the \lq\lq continuous" (according to Definition \ref{dcmn}) solutions which are the constants. 

Therefore, the fixed-point algebra is nontrivial ({\it cf.} Proposition \ref{toerp}), and ({\it cf.} Theorem \ref{classicalue})  
$(\bt\times\bt, R_{\theta, \varphi})$ is uniquely ergodic w.r.t. the fixed-point subalgebra.\footnote{Another example of a uniquely ergodic w.r.t. the fixed-point subalgebra $C^*$-dynamical system which is not uniquely ergodic arises from Quantum Probability, see \cite{CFL} Section 7, see also \cite{F5}, Section 5.2.}
\end{examp}

\begin{examp}
The situation described in Remark \ref{mov} occurs for instance in the following classical case.

Indeed, we consider the Anzai skew-product on $\bt\times\bt$ generated by $R_{\theta, -\th}\in\textrm{Homeo}(\mathbb{T}\times\mathbb{T})$ given in \eqref{drt}, with $\frac{\theta}{2\pi}$ being irrational. The corresponding cohomological equation for $n=1$ is 
$g(e^{\imath\theta}z)e^{-\imath\theta}=g(z)$, $z\in\mathbb{Z}$, and the function  
$g(z)=z$ is seen at once to be a continuous solution of it. Therefore, this dynamical system is uniquely ergodic w.r.t.
the fixed-point subalgebra. Interestingly, this is true for the quantum analogue of this system as well.

More precisely, denoting by $R_\th:z\mapsto e^{\imath\theta}z$ the rotation for the angle $\th\in[0,2\pi)$, we can consider the quantum version $\F_{R_\th,  e^{-\imath\theta}I}$ of the above classical flow on the noncommutative torus $\mathbb{A}_\alpha$. 

Since for $n=1$ the cohomological equation is just the same as the classical equation above, as pointed out in \cite{DFGR}, the $C^*$-dynamical system 
$\big(\mathbb{A}_\alpha,\F_{R_\th,  e^{-\imath\theta}I}\big)$ is uniquely ergodic w.r.t. the fixed-point subalgebra as well.
\end{examp}

\smallskip

In the remaining part, for the skew-product $\F_{\th,f}$ we use the notations borrowed from \cite{Fu, DFGR}, respectively, see Footnote 10. 

\begin{examp}
\label{ex0}
We deal with $X_o=\bz_\infty$, the one-point compactification of the integers $\bz$, together with the homeomorphism 
$\th:\bz_\infty\rightarrow\bz_\infty$ given by
$$
\th(l):=\left\{\begin{array}{ll}
                     l+1 & \!\!\text{if}\,\, l\in \bz\,, \\
                      \infty & \!\!\text{if}\,\, l=\infty\,.
                    \end{array}
                    \right.
$$
The dynamical system $(\bz_\infty,\th)$ is uniquely ergodic (but neither minimal, nor strictly ergodic) with the unique invariant measure
$$
\m_o(f):=f(\infty)\,,\quad f\in C(\bz_\infty)\,.
$$
We also note that 
$$
C(\bz_\infty)^\th:=\{f\in C(\bz_\infty)\mid f\circ\th=f\}=\bc 1\sim L^\infty(\bz_\infty,\m_o)\,,
$$
with $1$ being the function identically equal to one.

For $f\in C(\bz_\infty;\bt)$, we can associate the process on the torus $\F_{\th,f}:\bz_\infty\times\bt\rightarrow\bz_\infty\times\bt$ given by $(\F_{\th,f})(l,w)=(\th(l), f(l)w)$. We now particularise the situation for 
$$
f(l):=\left\{\begin{array}{ll}
                    \b & \!\!\text{if}\,\, l=0\,, \\
                      1 & \!\!\text{otherwise}\,,
                    \end{array}
                    \right.
$$
for some $\b\in\bt$.

It is immediate to show that for each $n\in\bz$, 
$$
\text{const}f(l)^n=\text{const}\,,\,\,\mu_o\,\text{-}\,\textrm{a.e.}\,,
$$
and therefore the \eqref{sifo} admit nontrivial solutions for each $n\in\bz$.

On the other hand, if the two-sided sequence $(g(l))_{l\in\bz}$ satisfies \eqref{sifo1} for some $n$, then
$$
g(l):=\left\{\begin{array}{ll}
                    g(0)\b^{-n} & \!\!\text{if}\,\, l>0\,, \\
                     g(0) & \!\!\text{if}\,\, l\leq0\,,
                    \end{array}
                    \right.
$$
Imposing continuity, we are led to
$$
g(0)=\lim_{l\to-\infty}g(l)=\lim_{l\to+\infty}g(l)=\b^{-n}g(0)\,.
$$

If $\b=e^{2\pi\imath\a}$ for some irrational  number $\a$, then the unique continuous solution of \eqref{sifo1} for some $n\neq0$ is accordingly the null sequence $g(l)=0$. Therefore, by Theorem \ref{mier}, 
$(X_o\times\bt, \F_{\th,f})$ provides a simple example of a (classical) $C^*$-dynamical system which is topologically ergodic but not uniquely ergodic.

If $\b=-1$, \eqref{sifo1} admits nontrivial solutions, provided $n$ is even, and therefore
we can exhibit a $C^*$-dynamical system which is neither topologically ergodic ({\it i.e.} $C(X_o\times\bt)^{\F_{\th,f}}\supsetneq \bc1$), nor uniquely ergodic w.r.t. the fixed-point subalgebra, see Theorem \ref{classicalue}. 

The above assertions can however be explicitly checked as follows.
Indeed, in both the above examples for which $f(l)=1$, $l\neq0$, we easily compute for the function $h(l,z)=z$, $l\in\bz,z\in\bt$,
$$
\lim_{n\to+\infty}\Big(\frac1{n}\sum_{k=0}^{n-1}h(\F_{\th,f}^k(l,z))\Big)=\left\{\begin{array}{ll}
                    z & \!\!\text{if}\,\, l>0\,, \\
                      f(0)z & \!\!\text{if}\,\, l\leq0\,,
                    \end{array}
                    \right.
$$
and therefore the Ces\`{a}ro averages of $h(l,z)$ do not converge to a function in $C(\bz_\infty,\bt)$ if $f(0)\neq1$.
\end{examp}

\begin{examp}

We provide as elementary an example as possible of a uniquely ergodic skew-product, where $\mathfrak{A}$ is nonetheless allowed to be 
highly non-commutative ({\it i.e.} simple). 

To this aim, we choose $\mathfrak{A}:=\mathbb{A_\gamma}$, the noncommutative 2-torus generated by two unitary indeterminate $U,V$ subjected to the relation
$UV=e^{2\pi\imath\g}VU$, where the deformation parameter $\gamma$
is assumed irrational. It is known that $\ba_\g$ is simple with a unique tracial state $\t$.

Corresponding to this choice of the initial $C^*$-algebra, we consider the
automorphism $\theta\in{\rm Aut}(\mathbb{A}_\gamma)$ whose action on the generators $U$ and $V$ is given by
$$
\theta(U)=e^{i\theta}U\quad{\rm and}\quad \theta(V)=UV\,.
$$
where $\theta\in\mathbb{R}$ is such that $\frac{\theta}{2\pi}$ is irrational.
In other words, $\theta$ is nothing but the quantum Anzai skew-product
$\Phi_{\theta, f}$, with $f(z)=z$, which is known to be uniquely ergodic with the canonical trace $\tau$ its unique
invariant state, see \cite{DFGR}, Proposition 4.4.

We then consider  a second automorphism $\alpha\in{\rm Aut}(\mathbb{A}_\gamma)$ given by
$$
\alpha(U)=U\quad{\rm and}\quad \alpha(V)=e^{i\alpha}V\,.
$$

Phrased differently, the automorphism $\alpha$ is nothing but a single rotation of an angle, denoted with an abuse of notation again by $\alpha$, w.r.t. the second generator.

We set $\mathfrak{B}:=\ba_\g \rtimes_{\alpha}\mathbb{Z}$, and denote by $W\in\mathfrak{B}$ any of the unitaries implementing $\alpha$, {\it i.e.} $WaW^*=\alpha(a)$, for any $a\in\mathbb{A}_\gamma$.

Since $\theta$ and $\alpha$ commute with each other and the centre of $\mathbb{A}_\gamma$ is trivial, the only possible choice for the intertwining unitary $u$ in \eqref{inter} is a scalar, say $u=\beta I$, with $\beta\in\mathbb{T}$. Using the notations in Section \ref{anspr}, the corresponding automorphism $\F_{\th,u}\in\aut(\ga\rtimes_\a\bz)$ will be simply denoted by $\F$.

Therefore, $\F$ will act on powers of $W$ as $\Phi(W^n)=\beta^n W^n$, which means
the cocycle $u_n$ is simply $\beta^n I$. The corresponding cohomological equations \eqref{compi}, $n\in\bz$, then read as 
\begin{equation}
\label{sifo2}
\beta^n\ad{}\!_{V_{\th,\t}}(x)=x\,,
\end{equation}
where the unknown $x\in\cam:=\pi_\tau(\mathbb{A}_\gamma)''$.
What we now want to do is show that, if $\beta$ is carefully chosen, then the above equations only have
trivial solutions for any $n\neq 0$.

To this end, rewrite any such an element as
$$
x=\sum_{k\in\bz}c_k(\pi_\t(U))\pi_\t(V)^k\,,
$$
where ({\it cf.}  \eqref{cesexp1}) $(c_k)_{k\in\bz}\subset L^\infty(\bt,\l)$, and the convergence is meant in the strong operator topology of $\cam$ in the sense of
Ces\`aro, see Proposition \ref{ftrf}.

With
$$
\d_k:=\left\{\begin{array}{ll}
                    e^{-\pi\imath k(k-1)} & \!\!\text{if}\,\, k>0\,, \\
                     1 & \!\!\text{if}\,\, k=0\,,\\
                      e^{-\pi\imath|k|(|k|+1)} & \!\!\text{if}\,\, k<0\,, 
                    \end{array}
                    \right.
$$
for fixed $n\in\bz$, the cohomological equation \eqref{sifo2} leads to
\begin{equation*}
\label{sifo3}
\b^n\d_kc_k(e^{\imath\th}z)z^k=c_k(z)\,,\,\,\l\,\text{-}\,\textrm{a.e.}\,, 
\end{equation*}
which should be satisfied for each $k\in\bz$.

We first fix $n=0$, and deduce by the Bessel inequality that the unique solution, up to a constant, of \eqref{sifo2} is $x=1$ ({\it i.e.} $c_k(z)=\text{const}\times\d_{k,0}$). This is in accordance with the fact that $(\ba_\g,\F_{\th,f})$ is uniquely ergodic, with the canonical trace as the unique invariant state.

Reasoning as before, for $n,k\neq0$ we can conclude that \eqref{sifo3} has only the null function as the solution, and therefore it remains open the case $n\neq0$, $k=0$.

By considering the Koopman operator $U$ associated to the rotation of the angle $\th$, and reasoning as in Example \ref{exax}, the eigenvalue equation 
will fail to have nontrivial solutions if and only if the sets $\{\beta^n\mid n\in\mathbb{Z}\}$ 
and $\{e^{\imath m\theta}\mid m\in\mathbb{Z}\}$ do not intersect each other. Writing $\beta=e^{\imath\varphi}$, we see this is the case if
and only if for any $m,n\in\mathbb{Z}$ the combination $m\frac{\theta}{2\pi}+n\frac{\varphi}{2\pi}$ is never integer, or in other words,
if $\frac{\theta}{2\pi}$ and $\frac{\varphi}{2\pi}$ are rationally independent.\footnote{A finite set $\{m_1,\dots,m_n\}\subset\br$, $n>1$, is said to be {\it rationally independent} if for any linear combination of the form
$\sum_{l=1}^nk_lm_l=0$, with $k_l$ integers, implies $k_1=k_2=\cdots=k_n=0$.} 

\end{examp}

\section*{Acknowledgements}
S.D.V. is supported by the Deutsche Forschungsgemeinschaft (DFG) within the Emmy Noether
grant CA1850/1-1. F.F.  acknowledges the \lq\lq MIUR Excellence Department Project'' awarded to the Department of Mathematics, University of Rome Tor Vergata, CUP E83C18000100006.

 We would also like to thank the anonymous referee(s), whose careful reading resulted in an improvement of the presentation of the paper.


\begin{thebibliography}{9999}


\bibitem{AD} Abadie B., Dykema K.
{\it Unique ergodicity of free shifts and some other automorphisms of $C^*$-algebras,}, J. Operator Theory {\bf 61} (2009), 279-294.

\bibitem {ACR} Aiello V., Conti R., Rossi S. \emph {A Fej\'er  theorem for boundary quotients arising from algebraic dynamical  systems},
Ann. Sc. Norm. Super. Pisa Cl. Sci. (5), {\bf XXII} (2021), 305-313.

\bibitem{A} Anzai H. 
{\it Ergodic Anzai skew-product transformations on the torus},
 Osaka Math. J. {\bf 3} (1951), 83-99

\bibitem{AKTH} Araki H., Kastler D., Takesaki M., Haag R. {\it Extension of KMS states and chemical potential},
Commun. Math. Phys. {\bf 53} (1977), 97-134.

\bibitem{BR} Bratteli O., Robinson D. W. {\it Operator algebras and quantum statistical mechanics 1}, Springer-Verlag, Berlin, 1987.

\bibitem{CH} Choda M. {\it Entropy of automorphisms arising from dynamical systems through discrete groups with amenable actions}, J. Funct. Anal. {\bf 217} (2004), 181-191.

\bibitem{CFL} Crismale V., Fidaleo F., Lu Y. G. {\it Ergodic theorems in
quantum probability: an application to the monotone stochastic processes},
Ann. Sc. Norm. Super. Pisa Cl. Sci. (5), \textbf{XVII} (2017),
113-141.

\bibitem{D} Davidson K. R. {\it $C^*$-algebras by example}, Providence RI, 1996.

\bibitem{DFGR} Del Vecchio S., Fidaleo F., Giorgetti L., Rossi S. {\it Ergodic properties of the Anzai skew-product for the noncommutative torus}, Ergod. Theory Dyn. Syst.
{\bf 41} (2021), 1064-1085.

\bibitem{DG} Del Vecchio S., Giorgetti, L. {\it Infinite Index Extensions of Local Nets and Defects}, Rev. Math. Phys. {\bf 30} (2018), 1850002.

\bibitem{DM} Dixmier J., Maréchal O. {\it Vecteurs totalizateurs d’une alg\`{e}bre de von Neumann}, Commun. Math. Phys. {\bf 22} (1971), 44-50.

\bibitem{F1} Fidaleo F. {\it On strong ergodic properties of quantum dynamical systems},
Infin. Dimens. Anal. Quantum Probab. 
Relat. Top. {\bf 12} (2009), 551-564.

\bibitem{F4} Fidaleo F. {\it On the uniform convergence of ergodic averages for $C^*$-dynamical systems}, Entropy,  
{\bf 20} (2018), 987 (9 pages).

\bibitem{F5} Fidaleo F. {\it Uniform convergence of Ces\`{a}ro averages for uniquely ergodic $C^*$-dynamical systems}, 
Mediterr. J. Math., {\bf 17} (2020), 125 (17 pages).

\bibitem{FidIso} Fidaleo F., Isola T. \emph{The canonical endomorphism for infinite index inclusion}, Z. Anal. Anwend.,
{\bf 18} (1999), 47-66.

\bibitem{FM3} Fidaleo F., Mukhamedov F. {\it Strict weak mixing of
some $C^*$-dynamical systems based on free shifts},
J. Math. Anal. Appl. {\bf 336} (2007), 180-187.

\bibitem{FZ} Fidaleo F., Zsido L. {Quantitative BT-Theorem and automatic continuity for standard von Neumann algebras}, Adv. Math. {\bf 289} (2016), 1236-1260.

\bibitem{Fu} Furstenberg H. {\it Strict ergodicity and transformation on the torus}, Amer. J. math. {\bf 83} (1961), 573-601.

\bibitem{Haa} Haagerup U.
{\it The standard form of von Neumann algebras},
Math. Scand. {\bf 37} (1975), 271-283.

\bibitem{LP} R. Longo, C. Peligrad \emph{Noncommutative topological dynamics and compact actions on 
$C^*$-algebras}, J. Funct. Anal. \textbf{58}  (1984), 157-174.

\bibitem{LR} Longo R., Rehren K. H. {\it Nets of Subfactors} Rev. Math. Phys. {\bf 7} (1995) 567-598.

\bibitem{MvN} Murray F. J., von Neumann J.
{\it On rings of operators}, Ann. Math. {\bf 37} (1936), 116-229.

\bibitem{NSZ} Niculescu C. P., Str\"oh A., Zsid\'o L.
{\it Noncommutative extension of classical and multiple recurrence 
theorems}, J. Operator Theory {\bf 50} (2003), 3-52.

\bibitem{OP} Osaka H., Phillips N. C. \emph{Furstenberg transformations on irrational rotation algebras},
Ergod. Theory Dyn. Syst. \textbf{26}  (2006), 1623-1651.
 
\bibitem{Sa} Sakai S. \emph{$C^*$-Algebras and $W^*$-Algebras}, Springer-Verlag, Berlin, 1971.

\bibitem{SZ} Str\v{a}til\v{a} \c{S}., L. Zsid\'o L. 
{\it Lectures on von Neumann algebras}, Abacus Press, Tunbridge
Wells, Kent, (1979).

\bibitem{T00} Takesaki M. {\it Duality for crossed products and the structure of von Neumann algebras of type $\ty{III}$},
Acta Math. {\bf 131} (1973), 249-310.

\bibitem{T0} Takesaki M. {\it Faithful states on a $C^*$-algebra},
Pacific J. Math. {\bf 52} (1973), 605-610.

\bibitem{T} Takesaki M. {\it Theory of operator algebras I}, Springer, Berlin-Heidelberg-New
York, 1979.

\bibitem{W} Walters P. {\it An introduction to ergodic theory}, Springer, Berlin-Heidelberg-New
York, 1982.


\end{thebibliography}
\end{document}